\documentclass[a4paper,10pt]{article}
\usepackage{a4wide}
\usepackage[tbtags,sumlimits,intlimits,namelimits,reqno,fleqn]{amsmath}
\usepackage{amsthm}
\usepackage{amsfonts}
\usepackage{latexsym}
\usepackage{amstext}
\usepackage{amssymb}
\usepackage{amsxtra}
\usepackage{authblk}
\usepackage{ifthen}   
\usepackage[english]{babel}
\usepackage{enumerate}
\usepackage{enumitem}   
\usepackage[pdfpagemode=UseOutlines,urlcolor=blue,colorlinks=true]{hyperref}
\usepackage[pdftex]{graphicx}
\usepackage[pdftex]{color}
\usepackage{bibentry}
\usepackage{natbib}
\usepackage{multirow}
\usepackage{relsize}
\usepackage{tikz}
\usetikzlibrary{patterns}

\bibpunct{(}{)}{;}{a}{}{,}

\numberwithin{equation}{section}

\newtheorem{Prop}{Proposition}[section]
\newtheorem{Lem}[Prop]{Lemma}

\newtheorem{Th}[Prop]{Theorem}

\newtheorem{Rm}[Prop]{Remark}

\newtheorem{Def}[Prop]{Definition}

\newtheorem{As}[Prop]{Assumption}
\newtheorem{Test}[Prop]{Test}

\newtheorem{Cor}[Prop]{Corollary}

\newtheorem{Ex}[Prop]{Example}

\newtheorem*{Def*}{Definition}
\numberwithin{equation}{section}

\newcommand{\iid}{\operatorname{\stackrel{i.i.d.}{\sim}}}

\newcommand{\inD}{\operatorname{\stackrel{\mathcal{D}}{\to}}}

\newcommand{\as}{\operatorname{~~a.s.}}

\newcommand{\cA}{\mathcal{A}}

\newcommand{\cC}{\mathcal{C}}
\newcommand{\cD}{\mathcal{D}}

\newcommand{\cN}{\mathcal{N}}

\newcommand{\cS}{\mathcal{S}}
\newcommand{\cR}{\mathcal{R}}

\newcommand{\Prb}{\mathbb{P}}

\newcommand{\CC}{\mathbb{C}}
\newcommand{\NN}{\mathbb{N}}

\newcommand{\EE}{\mathbb{E}}
\newcommand{\SSS}{\mathbb{S}}

\newcommand{\RR}{\mathbb{R}}

\DeclareMathOperator*{\argmin}{argmin}

\DeclareMathOperator*{\Cut}{Cut}

\DeclareMathOperator{\vol}{\mathbf{vol}}
\DeclareMathOperator{\RE}{\mathbf{Re}}
\DeclareMathOperator{\IM}{\mathbf{Im}}

\DeclareMathOperator{\tr}{\mathbf{tr}}

\DeclareMathOperator{\cov}{\mathbf{cov}}

\DeclareMathOperator{\grad}{\mathbf{grad}}

\DeclareMathOperator{\st}{\mathrm{such~that~}}

\newcommand{\sdp}{{\rtimes}}

\title{Two-Sample Tests for Optimal Lifts, Manifold Stability and Reverse Labeling Reflection Shape}

\author[1]{Do Tran Van\footnote{do.tranvan@uni-goettingen.de}}
\author[2]{Susovan Pal\footnote{Susovan.Pal@vub.be}}
\author[1]{Benjamin Eltzner\footnote{beltzne@uni-goettingen.de}}
\author[1]{Stephan F. Huckemann\footnote{huckeman@math.uni-goettingen.de}}
\affil[1]{\normalsize Institute for Mathematical Stochastics, University of G\"ottingen, Germany}
\affil[2]{\normalsize Mathematics and Data Science Group, Vrije Universiteit Brussel, Belgium}

\begin{document}

\maketitle
\date

\tableofcontents

\begin{abstract}
  We consider a quotient of a complete Riemannian manifold modulo an isometrically and properly acting Lie group and lifts of the quotient to the manifolds in optimal position to a reference point on the manifold. With respect to the pushed forward Riemannian volume onto the quotient we derive continuity and uniqueness a.e. and smoothness to large extents also with respect to the reference point. In consequence we derive a general manifold stability theorem: the Fr\'echet mean lies in the highest dimensional stratum assumed with positive probability, and a strong law for optimal lifts. This allows to define new two-sample tests utilizing individual optimal lifts which outperform existing two-sample tests on simulated data. They also outperform existing tests on a newly derived reverse labeling reflection shape space, that is used to model filament data of microtubules within cells in a biological application.
\end{abstract}

\section{Introduction}
    Statistical analysis of shape leads to the general task of performing statistics on data in a quotient, often a nonmanifold quotient $Q=M/G$ of a complete Riemannian manifold $M$ (of pre-shapes) modulo an isometric and proper action of a Lie group $G$ (conveying shape equivalence). The very successful and  highly popular \emph{Procrustes analysis} introduced by \cite{Gow} from \cite{HC62} registers pre-shape data to the tangent space of a pre-shape of a Procrustes mean and can then apply a plethora of statistical methods from the toolbox of Euclidean statistics \citep{DM16}. In the language of geometry this registration procedure amounts to \emph{lifting} the quotient $Q$ to the top manifold $M$ in \emph{optimal position} to some $p \in M$. For less concentrated data, due to the presence of curvature (shape spaces feature infinite positive curvatures near singularities, e.g. \cite{KBCL99,HHM07}) decisive shape information can get lost, however.
    
    In this contribution we aim at filtering out main effects of such curvatures caused by passing to the quotient. To this end, we first show that in general such lifts are measurable (Theorem \ref{th:exist-opt-measurable-lift}) and almost everywhere unique and continuous (Theorem \ref{th:cont-and-uniq-opt-lift}), even to large extents smooth, also with respect to their basepoints (Corollary \ref{cor:lift-smooth}). In passing, this allows to extend the \emph{manifold stability theorem} from \cite{H_meansmeans_12}, now stating that in general a Fr\'echet mean is assumed on the manifold part of the quotient, whenever it is assumed with positive probability (Theorem \ref{th:manifold-stability}). Hence, in combination with our smoothness Corollary \ref{cor:lift-smooth}, in general, for optimal lifts, central limit theorems on manifolds (ranging from \cite{BP05} to \cite{EltznereHuckemann2019} giving a rather general version including smeariness) can be applied.
    
    From a.e. continuity of optimal lifts in the basepoint we derive a strong law of large numbers for optimal lifts (Theorem \ref{th:SL-optimal-lift}). This is required for our new two-sample tests introduced in Section \ref{scn:tests} based on lifting each group separately, optimal with respect to optimally positioned lifts of their individual means, instead of lifting optimally with respect to a lift of the pooled mean (as is, e.g., done in Procrustes analysis). First simulations in Section \ref{scn:simulation-application} show that our new tests outperform classical tests.
    
    In view of an application in cell structure biology we introduce in Section \ref{scn:shape}  a new landmark based shape space modulo similarity and reflection and modulo \emph{reverse labeling} of landmarks, which is appropriate for modeling filament string-like shapes by placing landmarks  on them. Thus in Section \ref{scn:simulation-application}, our new tests can distinguish between two groups of cellular buckle structures which previous tests could not. 
    
    We note that the bias introduced by curvature due to passing to the quotient space has been corrected by \cite{MiolaneHolmesPennec2017}.

     Since our results build on a fair amount of geometry convolved with statistics, in the following section we recall a considerable body of foundations.

\section{Setup}

    Throughout this contribution we consider the following setup (detailed in standard textbooks, for instance \cite{Bre72,Lee1997,gray2003tubes,BB12,alexandrino2015lie,PatrangenaruEllingson2015,durrett2019probability}:
    \begin{enumerate}
     \item $M$ is a complete connected finite dimensional Riemannian manifold with countable atlas and intrinsic distance $d$,  tangent space $T_pM$  at $p\in M$, Riemannian inner product $\langle v,w\rangle$ and norm $\|v\|$ for $v,w\in T_pM$ and Riemannian exponential $\exp_p : T_pM \to M$, which has a well defined inverse $\exp_p^{-1}: U \to T_pM$ for some open set $p\in U \subseteq M$, $p$ in $M$. In particular, with the tangent bundle $TM = \cup_{p\in M}(\{p\}\times T_pM)$ 
     $$~ \exp : TM \to M,\quad (p,v) \mapsto \exp_pv$$
     is a smooth map. Here \emph{smooth} means that derivatives of any order exist.
     
     \item Notably, for an incomplete Riemannian manifold, the exponential is well defined and smooth in a neighborhood of the zero-section of the tangent bundle.

     \item For a closed set $A \subset M$ we have the distance $d(p',A) := \min_{p\in A}d(p',p)$ of a point $p'\in M$ from $A$.
     \item
     For a measurable set $A \subseteq M$ we have its (possibly infinite) Riemannian volume $\vol(A)$.
     
     \item For $p\in M$ we have its \emph{cut locus}
     \begin{eqnarray*} \Cut(p)&:=& \{\exp_p(v): v\in T_pM \st \\
      &&\hspace*{2cm}
d(\exp_p(tv),p) = t\|v\| \mbox{ for all }0\leq t \leq 1\mbox{ but }\\
      &&\hspace*{2cm}
d(\exp_p(tv),p) < t\|v\| \mbox{ for  al } t>1\}\,.      
     \end{eqnarray*}
     
 In order to use Theorem \ref{th:Le-Barden} by \cite{LeBarden2014} below, we assume that for every $p\in M$ and every $p' \in \Cut(p)$ there are $v_1,v_2 \in T_pM$ with $v_1 \neq v_2$ but $\exp_p(v_1) = p'= \exp_p(v_2)$. Manifolds where this is not the case are rather exotic.  
 
      For an incomplete, geodesically convex (see below) Riemannian manifold, the cut locus of a point $p$ is defined restricting $\exp_p$ to the largest open set in the tangent space for which $\exp_p$ is defined and smooth \citep[p. 117/118]{chavel1995riemannian}. 
      \item $G$ is a Lie group with unit element $e\in G$ acting on $M$ via $G\times M \to M, (g,p)  \mapsto g.p$, properly, i.e. preimages of compact sets under $G\times M \to M \times M, (g,p)  \mapsto (g.p,p)$ are compact, and isometrically, i.e. $d(p,p') = d(g.p,g.p')$ for all $p,p' \in M$ and $g\in G$.

     \item Then the quotient $Q:=M/G = \{[p]: p \in M\}$ with \emph{orbit} $[p] := \{g.p: g \in M\}$ is Hausdorff when $Q$ is equipped with the canonical quotient topology, i.e. the unique topology making the canonical projection $\pi: M \to Q, p \mapsto [p]$ continuous and open. Further,  for every $p\in M$, $[p]\subset M$ is a closed embedded submanifold and
      $$ d_Q([p],[p'])=\min_{g\in G}d(p,g.p')$$
      is the canonical quotients metric. We say that $p,p'\in M$ are \emph{in optimal position} if $d(p,p') = d_Q([p],[p'])$ \citep{HHM07}.

      For better clarity, from now on for $p \in M$, $[p]$ denotes the corresponding subset of $M$ which, as an element of $Q$ is denoted by $q$ such that $p\in [p] = \pi^{-1}(q)$.

     \item $I_p = \{g\in G: g.p = p\}$ is the \emph{isotropy group} at $p\in M$ and we assume that 
       $$ M^* := \{p \in M: I_p = \{e\}\} \neq \emptyset\,.$$
       Then, the  \emph{principal orbit theorem teaches that} $M^*$ is open and dense in $M$. 
\item
       For $p\in M$ every tangent space decomposes into a \emph{vertical space}  along the orbit $[p]$ and its orthogonal complement, the \emph{horizontal space} at $p\in M$:
       $$ T_pM = T_p[p] \oplus H_pM\,.$$
       Further, again by the principal orbit theorem, $Q^*:= M^*/G$ carries a canonical Riemannian quotient structure with $T_qQ^* \cong H_pM$ for every $p \in \pi^{-1}(q)$, $q\in Q^*$. For this reason, $Q^*$ is called the \emph{manifold part} of $Q$ and $Q^0= Q\setminus Q^*$ is the \emph{singular part}, with preimage  $M^0 = \pi^{-1}(Q^0)$.
       
       \item In particular, with the Riemann exponential $\exp^{Q^*}$ of $Q^*$ and the identification $T_{\pi(p)}Q^* \cong H_pM$ for $p \in M^*$,
       we have that
       $$ \exp_p|_{H_pM} = \exp^{Q^*}_{\pi(p)} \,$$
       whenever the r.h.s. is defined, as $Q^*$ is not complete in general. While only $Q$ is complete, $Q^*$ is, however, geodesically convex, i.e. for any two $q,q' \in Q^*$ there is a shortest curve in $Q$ joining the two and this curve lies fully in $Q^*$. In fact, with the parallel transport $\theta_{q,q'}: T_qQ^* \to T_{q'}Q^*$ along a length minimizing geodesic from $q'$ to $q$ (which is unique unless $q'\in \Cut(q)$) and $v\in T_qQ^*$ such that $q'= \exp_q^{Q^*}v$, we have
       $$ q = \exp_{q'}^{Q^*} (-\theta_{q,q'} (v))\,.$$ 

     \item $X_1,\ldots,X_n \iid X$ are random variables on $M$, i.e. they are Borel measurable mappings to $M$ from a silently underlying probability space $(\Omega, \Prb,\cA)$ where we denote by $\Prb^X(B) = \Prb\{X\in B\}$, for Borel measurable $B\subseteq M$, the \emph{push forward measure} on $M$. Here $\iid$ stand for identically and independently distributed, i.e.
     $\Prb^{(X_1,\ldots,X_n)}(B_1\times\ldots\times B_n) = \prod_{j=1}^n \Prb^{X_j}(B_j)$ for all Borel sets $B_j\in M$, $j=1,\ldots,n$ and identically distributed, i.e. $\Prb^{X_j} = \Prb^X$ for all $j=1,\ldots,n$.
     \item
     These random variables give rise to their \emph{population} and \emph{sample Fr\'echet functions}, respectively 
       $$ F^X(p) = \EE[d(p,X)^2],\quad F_n^X(p) := \frac{1}{n} \sum_{j=1}^n d(p,X_j)^2,\quad p \in M\,.$$ 
       Every minimizer is called a \emph{population} and \emph{sample Fr\'echet mean}, respectively
       $$ \mu \in \argmin_{p\in M} F^X(p),\quad  \mu_n \in \argmin_{p\in M} F^X_n(p)\,.$$ 
       \item 
       The above $X_1,\ldots,X_n \iid X$ project to random variables 
       $\pi \circ X_1,\ldots,\pi \circ X_n \iid \pi\circ X$ on $Q$ with the corresponding Fr\'echet functions 
       $$ F^{\pi\circ X}(a) = \EE[d(q,\pi \circ X)^2],\quad F_n^{\pi \circ X}(q) := \frac{1}{n} \sum_{j=1}^n d(q,\pi \circ X_j)^2,\quad q \in Q\,$$
       and Fr\'echet means
       $$ \nu \in  \argmin_{q\in Q} F^{\pi \circ X}(q),\quad  \nu_n \in \argmin_{q\in Q} F^{\circ X}_n(q)\,.$$ 

       \item Since every random variable $Y$ on $Q$ can be lifted to a random variable $X=\ell_p\circ Y$ as detailed in Theorem \ref{th:exist-opt-measurable-lift}, we can equivalently do our analysis starting with random variables on $M$, or starting with random variables on $Q$.
     \item We assume that $d_Q(q,\pi \circ X)$ is $\Prb$-integrable for some $q=q_0\in Q$, and hence, by the triangle inequality, for any $q\in Q$. Then $\emptyset\neq \argmin_{q\in Q} (F^{\pi \circ X}(q) -F^{\pi \circ X}(q_0))$, again by the triangle inequality, i.e., we assume that a population Fr\'echet mean exists.

    \end{enumerate}

\section{Manifold Stability and Optimal Lifts}

    The following main result of this section states that every Fr\'echet mean lies on the manifold part, if the latter is assumed with positive probability. This strengthens \citet[Corollary 1]{H_meansmeans_12} who assumed that there were only countably many point masses on the singular part. The rest of this section is devoted to its proof and a few more convenient results.
    
    \begin{Th}[Manifold Stability]\label{th:manifold-stability}
     Let $ \nu \in  \argmin_{q\in Q} F^{\pi \circ X}(q)$ and $\Prb\{\pi \circ X \in Q^*\} > 0$ then $$\nu \in Q^*$$
    \end{Th}

    \begin{proof} This is assertion (v) of Lemma \ref{lem:basic-lifts}.\end{proof}
   
    Now, we build the necessary  machinery.
For $p\in M$ let 
\begin{eqnarray*}
 L_p &:=&\{p' \in M: p' \mbox{ is in optimal position to } p\}\,,\\
 L'_p &:=&\{p' \in L_p: \mbox{ if } p'' \in [p'] \cap L_p \mbox{ then } p''=p'\}\,.
\end{eqnarray*}
Thus, $L'_p$ denotes all points, uniquely in optimal position to $p$.

\begin{Rm} Note that for all $p,p' \in M$, due to the proper  action, $L_p \cap [p'] \neq \emptyset$. Further, optimal positioning is symmetric, i.e.
 $$ p' \in L_p \Leftrightarrow p \in L_{p'}\,.$$
 Putting uniquely in optimal position, however, is not reflexive:
 $$ p\in M^*, p' \in L'_p\cap M^0 \Rightarrow p \not \in L'_{p'}\,.$$
 On the preimage of the manifold part, however, it is
 $$ p\in M, p' \in L'_p\cap M^* \Rightarrow p \in L'_{p'}\,.$$
\end{Rm}

\begin{Ex}\label{ex:1} Consider  on $M= \RR^3$ the action of $G$ of rotations about the last axis $N = \{(0,0,z)^T: z \in \RR\}$ and the points
 $$ p_1 = \left(\begin{array}{c} 0 \\ 0 \\ 1 \end{array}\right),\quad p_2 = \left(\begin{array}{c} 1 \\ -1 \\ 1 \end{array}\right),\quad p_3 = \left(\begin{array}{c} 2 \\ - 2 \\ 3 \end{array}\right),\quad p_4 = \left(\begin{array}{c} -2\\ 2\\ 3\end{array}\right)\,.$$
 Then $p_2,p_3,p_4 \in L_{p_1}, p_1,p_3 \in L'_{p_2} \subset L_{p_2}$ but $p_2\not\in L'_{p_1} = N$.
\end{Ex}

        \begin{Th}[Existence of optimal lifts]\label{th:exist-opt-measurable-lift} For every $p\in M$ there is a measurable mapping $\ell_p : Q \to M$ such that
         \begin{enumerate}[label=(\roman*)]
          \item $\pi\circ \ell_p = id_Q$  (lift),
          \item $\ell_p(q) \in L_p$ for all $q\in Q$  (optimal). 
         \end{enumerate}
        \end{Th}

  \begin{proof}
   Let $p \in M$. By continuity, $L_p$ is closed and hence so is $\emptyset \neq L_p \cap \pi^{-1}(q) \subseteq M$ for all $q\in Q$. Setting
   \begin{eqnarray*}
    \Psi : Q &\to& \{  \emptyset \neq A \subset M: A \mbox{ closed}\}\\
    q &\mapsto & L_p \cap \pi^{-1}(q)\,,
   \end{eqnarray*}
   apply \cite[Theorem 6.9.3]{Bog07-2} stating that there exists a measurable map $\xi: Q \to M$ ($M$ is called $X$ in \cite{Bog07-2} and $Q$ is called $\Omega$ there) with $\xi(q) \in \Psi(q)$, if 
   $$ \{q\in Q: \Psi(q) \cap U \neq\emptyset\}\mbox{ is measurable for all open }U\subseteq M\,.$$
   Indeed, this condition is met, since for open $U\subseteq M$, $W:=L_p \cap U$ is measurable in $M$, and hence, so is
   \begin{eqnarray*}
   \{q\in Q:  \Psi(q) \cap U \neq\emptyset\} &=& \pi\{p' \in M:  W\cap \pi^{-1}(\pi(p')) \neq \emptyset)\} ~=~   \pi(W)\,.
   \end{eqnarray*}
   Thus $\ell_p := \xi$ is a measurable optimal lift.
  \end{proof}

    \begin{Def}
     Every such $\ell_p$ from Theorem \ref{th:exist-opt-measurable-lift} is called an \emph{optimal lift} of $Q$ \emph{through} $p\in M$ (we drop the adjective  'measurable').
    \end{Def}

    The following theorem states that Fr\'echet means are equilibrium points, i.e. classical means in the inverse exponential chart. 
    
    \begin{Th}[\cite{LeBarden2014}, Corollaries 2 and 3]\label{th:Le-Barden}
     If $\mu \in \argmin_{p\in M}F^X(p)$ then
     \begin{enumerate}[label=(\roman*)]
      \item the domain $0\in V\subset  T_\mu M$ of $\exp^{-1}_\mu$ can be chosen such that $\Prb\{X \in \exp_\mu V\} = 1$ and 
      \item $\int_M \exp^{-1}(p) d\Prb^X(p) =0$.
    \end{enumerate}

    \end{Th}

    We are now ready for the toolbox of this section.
    
    \begin{Lem}\label{lem:basic-lifts}
     The following hold:
     \begin{enumerate}[label=(\roman*)]
      \item $F^{\ell_p\circ \pi \circ X}(p) = F^{\pi\circ X}(\pi(p))$ for all $p\in M$
     \end{enumerate}
     Moreover, if $\nu \in \argmin_{q\in Q}  F^{\pi\circ X}(q)$, $\mu \in \pi^{-1}(\nu)$ and optimal lift $\ell_\mu$ through $\mu$, then
     \begin{enumerate}[label=(\roman*)] \setcounter{enumi}{1}
      \item $\mu \in  \argmin_{p\in M}  F^{\ell_\mu \circ \pi\circ X}(p)$
      \item if $p' \in  \argmin_{p\in M}  F^{\ell_\mu \circ \pi\circ X}(p)$ then $\pi(p') \in \argmin_{q\in Q}  F^{\pi\circ X}(q)$
      \item if $\ell'_\mu$ is another optimal lift through $\mu$ then $\ell'_\mu = \ell_\mu$ a.s.
      \item if $\Prb\{\pi \circ X \in Q^*\} > 0$ then $\nu \in Q^*$
      \item if $\{\nu\} = \argmin_{q\in Q}  F^{\pi\circ X}(q)$ and $\Prb\{X \in L'_\mu\} >0$  then $\{\mu\} = \argmin_{p\in M}  F^{\ell_\mu \circ \pi\circ X}(p)$
     \end{enumerate}     
    \end{Lem}

\begin{proof}
 (i) follows from $d(p,\ell_p(q')) = d_Q(q,q')$ for arbitrary $p\in M$ with $q = \pi(p)$ and from
 \begin{eqnarray*}
   F^{\ell_\mu \circ \pi\circ X}(p)&=& \int_M d(p,p')^2\, d\Prb^{\ell_\mu \circ \pi\circ X}(p')\\
   &=& \int_M d(p,\ell_p(q'))^2\, d\Prb^{\pi\circ X}(q') ~=~ F^{\pi\circ X}(\pi(p))\,.
 \end{eqnarray*}

 (ii) follows at once from (i).
 
 (iii): if $p' \in  \argmin_{p\in M}  F^{\ell_\mu \circ \pi\circ X}(p)$ then, due to (ii) and twice (i), 
 \begin{eqnarray*}
 F^{\ell_\mu \circ \pi\circ X}(p') &=&  F^{\ell_\mu \circ \pi\circ X}(\mu)~=~F^{\pi \circ X}(\pi(\mu))\\
 &\leq & F^{\pi \circ X}(\pi(p')) ~=~ F^{\ell_{p'} \circ \pi\circ X}(p')\\
 &=& \int_Q d(p'\ell_{p'}(q))^2 \, d\Prb^{\pi \circ X}\\
 &\leq &\int_Q d(p'\ell_{\mu}(q))^2 \, d\Prb^{\pi \circ X} ~=~F^{\ell_\mu \circ \pi\circ X}(p')\,. 
 \end{eqnarray*}
 Hence, all inequalities are equalities and thus $\pi(p') \in  \argmin_{q\in Q}  F^{\pi\circ X}(q)$.
 
 (iv): Let $\ell_\mu$ and $\ell'_\mu$ be two optimal lifts through $\mu$. Then, due to Theorem \ref{th:Le-Barden} and (ii),
 \begin{eqnarray}\label{eq:proof1-lem-basic-lifts} \nonumber
 \int_Q \exp_\mu^{-1}\circ \ell_\mu(q)\, d\Prb^{\pi\circ X}(q) &=& \int_M \exp_\mu^{-1}(p)\, d\Prb^{\ell_\mu \circ \pi\circ X}(p) ~=~0\\
 &=& \int_M \exp_\mu^{-1}(p)\, d\Prb^{\ell'_\mu \circ \pi\circ X}(p) ~=~\int_Q \exp_\mu^{-1}\circ \ell'_\mu(q)\, d\Prb^{\pi\circ X}(q)\,.
 \end{eqnarray}
 Now, in order to see (iv), it suffices to show that 
  \begin{eqnarray}\label{eq:proof2-lem-basic-lifts} 
   \int_B \exp_\mu^{-1}\circ \ell_\mu(q)\, d\Prb^{\pi\circ X}(q) &=& \int_B \exp_\mu^{-1}\circ \ell'_\mu(q)\, d\Prb^{\pi\circ X}(q)
  \end{eqnarray}
   for all Borel $B\subseteq Q$ and to this end, fix such a $B$ and define 
   $$ \ell''_\mu(q) = \left\{\begin{array}{rcl}
                              \ell_\mu(q)&\mbox{ if }& q\in B\\
                              \ell'_\mu(q)&\mbox{ if }& q\in Q\setminus B
                             \end{array}\right.\,,
  $$
  which is yet another optimal lift through $\mu$.
 Then, due to (\ref{eq:proof1-lem-basic-lifts}),
 \begin{eqnarray*}
  \int_{Q\setminus B} \exp_\mu^{-1}\circ \ell_\mu(q)\, d\Prb^{\pi\circ X}(q) + \int_B \exp_\mu^{-1}\circ \ell_\mu(q)\, d\Prb^{\pi\circ X}(q) &=& \int_Q \exp_\mu^{-1}\circ \ell_\mu(q)\, d\Prb^{\pi\circ X}(q) \\
  &=& \int_B \exp_\mu^{-1}\circ \ell''_\mu(q)\, d\Prb^{\pi\circ X}(q)\\
  ~=~\int_{Q\setminus B} \exp_\mu^{-1}\circ \ell_\mu(q)\, d\Prb^{\pi\circ X}(q) + \int_B \exp_\mu^{-1}\circ \ell'_\mu(q)\, d\Prb^{\pi\circ X}(q) \,.
 \end{eqnarray*}
  Subtracting $\int_{Q\setminus B} \exp_\mu^{-1}\circ \ell_\mu(q)\, d\Prb^{\pi\circ X}(q)$ from both sides yields the desired (\ref{eq:proof2-lem-basic-lifts}).

 (v): Assume the contrary, namely that $ \nu \in  \argmin_{q\in Q} F^{\pi \circ X}(q)$ and $\Prb\{\pi \circ X \in Q^*\} > 0$ but $\nu \in Q\setminus Q^*$, i.e. $\mu \in M^0$ for any $\mu \in \pi^{-1}(\nu)$ and we show a contradiction. Then, there is $e\neq g\in G$ with $g.\mu = \mu$ but $g.p \neq p$ for all $p \in M^*$. Moreover, then, consider an optimal lift $\ell_\mu$ through $\mu$ and note that $\ell'_\mu = g.\ell_\mu$, due to $d(\mu,g.\ell_\mu(p')) =  d(g^{-1}.\mu,\ell_\mu(p')) = d(\mu,\ell_\mu(p))$ for all $p\in M$ by virtue of isometric action, is also an optimal lift through $\mu$. Hence, by hypothesis and (iv),
 \begin{eqnarray*}
  0 ~<~ \Prb\{\pi \circ X \in Q^*\}  &=& \Prb^X(M^*) \\
  &\leq &\Prb^X\{p \in M: g.p\neq p\} ~=~\Prb^{\pi \circ X}\{q\in Q: \ell_\mu(q) = \ell'_{\mu}(q)\} ~=~0\,,
 \end{eqnarray*}
  the desired contradiction.
 
 (vi): Let $\{\nu\} = \argmin_{q\in Q}  F^{\pi\circ X}(q)$ and $\mu,\mu' \in  \argmin_{p\in M}  F^{\ell_\mu \circ \pi\circ X}(p)$. Due to (iii) $\mu,\mu' \in \pi^{-1}(\nu)$, i.e. there is $g \in G$ with $\mu'=g.\mu$. Hence
 $$\int_M d(\ell_\mu(p),\mu)^2\,d\Prb^{\ell_\mu \circ \pi\circ X}(p) = F^{\ell_\mu \circ \pi\circ X}(\mu) = F^{\ell_\mu \circ \pi\circ X}(g.\mu) = \int_M d(\ell_\mu(p),g.\mu)^2\,d\Prb^{\ell_\mu \circ \pi\circ X}(p)\,. $$ 
Since $d(\ell_\mu(p),\mu) \leq d(\ell_\mu(p),g.\mu)$ for all $p\in M$, in consequence of $a^2-b^2 = (a+b)(a-b)$ we have
 $$\Prb^{\ell_\mu \circ \pi\circ X}\{p \in \ell_\mu(Q): d(p,\mu) = d(p,g.\mu)\} =1\,. $$
 Thus, by hypothesis,
 $$\Prb^{\ell_\mu \circ \pi\circ X}\{p \in \ell_\mu(Q) \cap L'_\mu: d(p,\mu) = d(p,g.\mu)\} >0\,, $$
 yielding that there is $p \in \ell_\mu(Q) \cap L'_\mu$ with $d(p,\mu) = d(g^{-1}.p,\mu)$, i.e. $g^{-1}.p =p$, hence $g=e$ and thus $\mu'=g.\mu = \mu$.
\end{proof}

\section{Smoothness of Optimal Lifts}\label{scn:smooth}

\begin{Def} Let $P\subset M$ be a closed submanifold with \emph{normal bundle} $NP$, a subbundle of $TM$. Then  the cut locus of $P$ is defined as
     \begin{eqnarray*} \Cut(P)&:=& \{\exp_p(v): (p,v) \in NP \st \\
      &&\hspace*{2cm}
d(\exp_p(tv),P) = t\|v\| \mbox{ for all }0\leq t \leq 1\mbox{ but }\\
      &&\hspace*{2cm}
d(\exp_p(tv),P) < t\|v\| \mbox{ for  all } t>1\}\,.     
     \end{eqnarray*}
      \end{Def}

      Since $\pi^{-1}(q)$ is a closed submanifold of $M$ for every $q\in Q$, with its normal bundle $N\pi^{-1}(q)$, we have a smooth mapping 
    \begin{eqnarray}\label{eq:Phiq} \Phi_{q}: N\pi^{-1}(q) \to M\, \quad (p,v) &\mapsto& \exp_{p}v\,.
    \end{eqnarray}
      Note that $(p,v)\in N\pi^{-1}(q)$ if and only if $v\in H_{p}M$.

      \begin{Th}[\cite{basu2023connection}, p. 4218]\label{th:basu}
       For every $q \in Q$ the inverse map $\Phi_q^{-1}$ is well defined and smooth on $M\setminus (\Cut(\pi^{-1}(q))\cup \pi^{-1}(q))$.
      \end{Th}
      
      \begin{Rm}\label{rm:th-basu} Note that $\Phi_q^{-1}$ can be continuously extended to any $p\in \pi^{-1}(q))$ via $\Phi_q^{-1}(p) = (p,0)$, since for $\pi^{-1}(q))\not \ni p'_n \to p\in \pi^{-1}(q))$ and $\Phi_q^{-1}(p'_n) = (p_n,v_n)$ with $d(p'_n,p_n) = \|v_n\| \leq d(p'_n,p) \to 0$. In fact, this extension is smooth by \citet[Lemma 2.3]{gray2003tubes}.
    \end{Rm}
      
      \begin{Cor}\label{cor:lift-smooth} 
      Let $q\in Q$ and $p'\in M\setminus \Cut(\pi^{-1}(q))$ with $\Phi^{-1}(p') = (p,v)$. Then
       $$p' \in L_{p}\mbox{ and }p\in L'_{p'}\,.$$
       Moreover, for every open $U \subseteq Q$ and $p' \in M \setminus C$ where $C:=\cup_{q\in U} \Cut(\pi^{-1}(q))$, there is an optimal lift through $p'$ of form
       $$ \ell_{p'}(q)= \exp_pv \mbox{ for }q\in U,\quad \mbox{ where }\Phi_q^{-1}(p') = (p,v)\,.$$
       Further, this lift varies smoothly in $p'\in M\setminus C$ and also in $q\in U$ if $U\subseteq Q^*$ and $p' \in M^*\setminus C$.
      \end{Cor}

      \begin{proof} 
       By definition, if $p'= \exp_pv \in M\setminus (\Cut(\pi^{-1}(q))\cup \pi^{-1}(q))$ with $\Phi^{-1}(p')= (p,v)$ then $d(p',\pi^{-1}(q)) = \|v\| = d(p',p)$ so that $p' \in L_{p}$. If also $g.p \in L_{p'}$ for $e\neq g\in G$ then there would be two geodesics from $\pi^{-1}(q)$ to $p'$, a contradiction to $p' \not\in \Cut(\pi^{-1}(q))$.
       
       The second assertion follows at once from the first assertion by fixing $p' \in M\setminus C$, taking any optimal lift through $p'$ and modifying it on the measurable set $U$.
       
       For the last assertion, smoothness in $p' \in M\setminus C$ follows from Theorem \ref{th:basu} and Remark \ref{rm:th-basu}. 
       
       If $U \subseteq Q^*$, since $Q^*$ is a manifold, due to geodesic convexity of $Q^*$, there is $v\in T_qM$, varying smoothly in $q\in Q^*$, such that $q'=\pi(p') = \exp_q^{Q^*}v$ and hence
       $$ q = \exp_{q'}^{Q^*} (-\theta_{q,q'} (v))$$
       where $\theta_{q,q'}: T_qQ^* \to T_{q'}Q^*$ is the parallel transport which is smooth. Similarly, for $p \in L'_{p'}$ (see first assertion), $p' = \exp_{p}v$, where we have identified $H_pM \cong T_qQ^*$, i.e. 
       $$ \ell_{p'}(q) = p = \exp_{p'} (-\theta_{q,q'} (v))$$
       which varies smoothly in $q\in U \subseteq Q^*$.
      \end{proof}

%
%
%
%
%

\section{The Strong Law for Optimal Lifts}

\begin{Th}[The Strong Law for Optimal Lifts]\label{th:SL-optimal-lift} Let $X$ be absolutely continuous w.r.t. $\vol$ and with unique Fr\'echet mean
$$\{\nu\} = \argmin_{q\in Q} F^{\pi \circ X}(q)\mbox{  and measurable selection }\nu_n \in \argmin_{q\in Q} F_n^{\pi \circ X}(q)\,.$$
Then, for arbitrary 
$$\mu \in \pi^{-1}(\nu)\mbox{ and measurable selection }\mu_n \in \pi^{-1}(\nu_n) \cap L_\mu$$
and arbitrary optimal lifts $\ell_\mu$ and $\ell_{\mu_n}$ through $\mu$ and $\mu_n$, respectively, we have that
$$\ell_{\mu_n} \to \ell_\mu\as$$
in the sense that
$$ \int_Q d(\ell_{\mu_n}(q),\ell_\mu(q))\, d\Prb^{\pi \circ X}(q) \to 0 \as\,.$$
\end{Th}

\begin{proof} Since $\Prb^X$ is absolutely continuous w.r.t. to $\vol$, we have $0<\Prb\{X \in M^*\} \leq \Prb\{\pi\circ X \in Q^*\}$ and hence by manifold stability (Theorem \ref{th:manifold-stability}) that $\nu \in Q^*$, i.e. $\mu \in M^* \subseteq F$ ($F$ is introduced below in Definition \ref{def:CS-F} and only needed to apply Lemma \ref{lem:F_p_advanced} below), due to Theorem \ref{th:F=Mstar}, and similarly $\nu_n \in Q^*\as$, i.e. $\mu_n \in M^* \subseteq F\as$. 
Then, by (ii) and (iii) of  Lemma \ref{lem:F_p_advanced},
\begin{eqnarray}\label{eq:proof1-SL-optimal-lift}
 d(\ell_{\mu_n}(q),\ell_\mu(q)) \to 0 &\mbox{ for all }&q\in Q\quad \Prb^{\pi \circ X}\mbox{-a.e.}
\end{eqnarray} 
holds a.s., since by Ziezold's strong law for Fr\'echet means on metric spaces \citep{Z77}, $\nu_n\to  \nu$ a.s. and hence $d(\mu_n,\mu) = d_Q(\nu_n,\nu)\to 0\as$.

Moreover, observe
\begin{eqnarray*}
 d(\ell_{\mu_n}(q),\ell_\mu(q)) &\leq& d(\ell_{\mu_n}(q),\mu_n) + d(\mu_n,\mu)+ d(\mu,\ell_\mu(q))\\
 &=& d_Q(q,\nu_n) + d(\mu_n,\mu)+ d_Q(\nu,q)\\
 &\leq& 2d(\mu_n,\mu)+ 2d_Q(\nu,q)\,.
\end{eqnarray*}
Since for fixed $\mu_n$, the bottom row as a function of $q$ is $\Prb^{\pi \circ X}$ integrable (by hypothesis a Fr\'echet mean exists), by dominated convergence, from  (\ref{eq:proof1-SL-optimal-lift}) infer that indeed
$$ \int_Q d(\ell_{\mu_n}(q),\ell_\mu(q))\, d\Prb^{\pi \circ X}(q) \to 0 \as\,.$$
\end{proof}

\begin{Def}\label{def:CS-F}
 For closed  $S \subset M$ define
 $$C_S:= \{p \in M: \mbox{ there are } S\ni s_1\neq s_2 \in S\st d(p,s_1)= d(p,S)=d(p,s_2)\}\,.$$
 Further, we set for $p \in M$,
 \begin{eqnarray*}
  F_p&:=&\bigcup_{p' \in L'_p}[p'] ~=~\{p' \in M:  \mbox{ there is a uniuqe }s'\in [p']\mbox{ with }d(p,s)=d(p,[p'])\}\\
  F&:=&\{p \in M: \vol(M\setminus F_p)=0\}
 \end{eqnarray*}
\end{Def}

\begin{Lem}\label{lem:foci-closed-set}
 Let $S \subset M$, then $\vol(C_S) = 0$.
\end{Lem}

\begin{proof}
Claim I: $f:M \to [0,\infty)$, $p\mapsto d(p,S)$ is $1$-Lipschitz.

Indeed, for $p_1,p_2\in M$ with $s_1,s_2\in S$ such that $d(p_i,S) = d(p_i,s_i)$, $i=1,2$, we have by the triangle inequality,
\begin{eqnarray*}
  f(p_1) - f(p_2) &=& d(p_1,s_1)-d(p_2,s_2)~\leq~d(p_1,s_2)-d(p_2,s_2)~\leq~d(p_1,p_2)\\
 -\left( f(p_1) - f(p_2)\right) &=& d(p_2,s_2)-d(p_1,s_1)~\leq~d(p_2,s_1)-d(p_1,s_1)~\leq~d(p_1,p_2)\,,
\end{eqnarray*}
yielding Claim I.

Claim II: $\grad_p f$ exists $\vol$-a.e. for $p$ with $\|\grad_p f\|\leq 1$ there.

In every local chart $(U,\exp^{-1}_{p})$, the Riemannian volume pushes forward to a measure absolutely continuous with respect to Lebesgue measure in $\exp^{-1}_{p}(U) \subseteq T_pM$. Making $U$ small enough, the Euclidean distance is bounded by a multiple of the intrinsic distance of preimages. Hence the intrinsic distance in $U$ to $U\cap S$ is Lipschitz in terms of Euclidean distance, and thus by Rademacher's theorem (e.g. \citet[Theorem 3.1]{heinonen2005lectures}) the subset in $\exp^{-1}_{p}(U)$ where $x \mapsto d(\exp_p x,S)$ is not differentiable, has Lebesgue measure zero. In consequence, $\vol(S_U)=0$ for the subset $S_U$ in $U\cap S$ where $p'\mapsto d(p',S)$ is not differentiable. Thus the first assertion of Claim II follows since $M$ has a countable atlas.

For the second assertion, consider $p\in M$ where $\grad_p f$ is defined, and arbitrary $v\in T_pM$. Then, due to Claim I,
\begin{eqnarray*}
 |\langle \grad_pf,v\rangle| &=& \left| \frac{d}{dt} f(\exp_p(tv))\right| ~=~ \lim_{t\to 0}\frac{1}{t}\left|f(\exp_p(tv)) - f(p)\right|\\
 &\leq &  \lim_{t\to 0}\frac{1}{t} d(\exp_p(tv),p) ~=~\|v\|\,,
\end{eqnarray*}
yielding the second assertion.

Claim III: If $p\in C_S$ then $f$ is not differentiable at $p$.

For otherwise, if $f$ was differentiable at $p \in  C_S$, there would be $S\ni s_1\neq s_2 \in S$ with $d(s_i,p) = f(p)$ for $i=1,2$. In particular, since $M$ is complete, 
there would be two different unit speed geodesics $\gamma_i: t\mapsto \exp_p (tv_i)$ with $T_pM \ni v_1 \neq v_2\in T_pM$, $\|v_1\| = 1 = \|v_2\|$ and $s_i = \exp_p(f(p)v_i)$ for $i=1,2$. Thus, in particular, 
$$f(\gamma_i(t)) = d(\gamma_i(t),s_i) = f(p) - t,\mbox{ for all }0 \leq t\leq f(p)$$
In consequence
$$ - 1 = \frac{d}{dt}\Big|_{t=0}f(\gamma_i(t)) = \langle \grad_p f,v_i\rangle,\mbox{ for }i=1,2\,,$$
which, in conjunction with the second part of Claim II ($\|\grad_p f\|\leq 1$), yields that $\grad_p f = -v_i$ for $i=1,2$, a contradiction.

Hence, due to Claim III, $C_S \subseteq \{p \in M: f \mbox{ is not differentiable at }p\}$, where the latter has zero volume, due to the first part of Claim II. Thus $\vol(C_S)=0$ as asserted.
\end{proof}

\begin{Rm}\label{rm:foci-closed-set} Note that the assertion of Lemma \ref{lem:foci-closed-set} can be wrong if $M$ ceases to be complete (completeness was essential for proving Claim III above). For instance, with $M=\RR^2\setminus \{(x,0)^T: x>0\}$ and $S = \{(\cos\phi,\sin\phi)^T: \phi \in [\pi/4,\pi/2]\}$ verify that $C_S = \{(x,y)^T\in M: x\geq 0  \geq y\}$ which has infinite volume, see Figure \ref{fig:foci-closed-set}.
\end{Rm}

\begin{figure}
    \centering
    \includegraphics[width=0.3\linewidth]{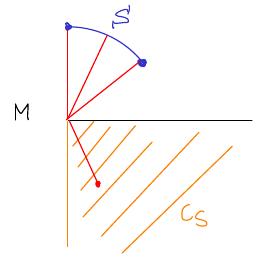}
    \caption{\it Depicting the counterexample from Remark \ref{rm:foci-closed-set}: All points on the circular segment $S$ are closest to any point in the lower right quadrant which is  $C_S$ in the incomplete $M$ which is the plane minus a slit along the positive first axis.}
    \label{fig:foci-closed-set}
\end{figure}

\begin{Lem}\label{lem:F_p_basics}
 The following hold:
 \begin{enumerate}[label=(\roman*)]
  \item Let $p,p' \in M$, then
  $ p' \in F_p \Leftrightarrow p \in M\setminus C_{[p']}$,
  \item $\pi^{-1}\circ\pi (F_p) = F_p$ for all $p\in M$,
  \item $\vol(M\setminus F_p) =0$ for $p\in M$ $\vol$-a.e.,
  \item $\vol(M\setminus F) =0$.
 \end{enumerate}
\end{Lem}

\begin{proof} (i) follows directly from definition.
 
 (ii): The inclusion $F_p \subseteq \pi^{-1}\circ\pi (F_p)$. To see the reverse, let $p'\in \pi^{-1}\circ\pi (F_p)$. Then $\pi(p') \in \pi(F_p)$, and hence, there are $g,g'\in G$ with $g'.p' \in g.F_p$, i.e. $g^{-1}g'.p' \in F_p$. Hence, there is a unique $p'' \in [p']$ such that $p''$ is in optimal position to $p$, hence $p' \in F_p$, yielding the reverse inclusion  $\pi^{-1}\circ\pi (F_p)\subseteq F_p $ and thus equality.
 
 (iii): Since $[p']\subset M$ is closed, $\vol(C_{[p']})=0$ by Lemma \ref{lem:foci-closed-set} for all $p'\in M$. Hence,
 $$ 0 = \int_M\left(\int_{C_{[p']}} d\vol(p)\right) d\vol(p') = \int_M \left(\int_{M\setminus F_p} d\vol(p')\right) d\vol(p)$$
 by Fubini and (i), yielding $\vol(M\setminus F_p) =0$ for $p\in M$ $\vol$-a.e. which is (iii).
 
 (iv) follows in consequence of the above, as,
 $$ 0 = \vol\{p \in M: \vol(M\setminus F_p) \neq 0\} = \vol(M\setminus F)\,.$$

\end{proof}

\begin{Lem}\label{lem:F_p_advanced}
Let $\vol_\sharp = \vol \circ \pi^{-1}$ be the pushforward of the Riemannian volume on $M$ to $Q$. Then
 \begin{enumerate}[label=(\roman*)]
  \item $\vol_\sharp(Q\setminus \pi(F_p)) =0$  for all $p\in F$.
  \end{enumerate}
 Further assuming that $F\ni p_n \to p \in F$ and setting $\tilde Q:= \pi(F_p) \cap \bigcap_{n=1}^\infty  \pi(F_{p_n})$, we have 
 \begin{enumerate}[label=(\roman*)] \setcounter{enumi}{1}
  \item $\vol_\sharp(Q\setminus \tilde Q) =0$, 
\item $\ell_{p_n}(q) \to \ell_p(q)$ for all $q\in \tilde Q$; here $\ell_{p_n}$ and $\ell_p$ are arbitrary optimal lifts through $p_n$ and $p$, respectively.
  \end{enumerate}
\end{Lem}

\begin{proof} (i): From the definition and (ii) of Lemma \ref{lem:F_p_basics} we have
 $$ vol_\sharp(Q\setminus \pi(F_p)) = \vol(M \setminus F_p)\,.$$
 where, by definition, the above r.h.s. is zero for all $p\in F$.
 
 (ii): By de-Morgan's law:
  $$ vol_\sharp\left(Q\setminus \tilde Q\right) = \vol_\sharp\left(\big(Q\setminus \pi(F_p)\big) \cup \bigcup_{n=1}^\infty  \big(Q\setminus \pi(F_{p_n})\big)\right)\,, $$
  whose r.h.s. is zero by (i) since $p, p_n\in F$ for all $n\in \NN$.
  
  (iii): Since
  \begin{eqnarray*}
   d(\ell_{p_n}(q),\ell_p(q)) &\leq&  d(\ell_{p_n}(q),p_n) + d(p_n,p) + d(p,\ell_p(q))\\
   &=& d(\pi^{-1}(q),p_n) + d(p_n,p) + d(p,\pi^{-1}(q))\\
   &\to&2d(p,\pi^{-1}(q))\,,
  \end{eqnarray*}
   observe that $\ell_{p_n}(q)$ is bounded. Since $M$ is Heine-Borel (by hypothesis), $\ell_{p_n}(q)$ has a cluster point $p' \in \pi^{-1}(q)$. Now suppose that $p'' \in \pi^{-1}(q)$ is another cluster point of $\ell_{p_n}(q)$.  Since $d(p_n, \ell_{p_n}(q)) = d(p_n,\pi^{-1}(q)) \to d(p,\pi^{-1}(q))$ we have thus
   $$ d(p,p') = d(p,\pi^{-1}(q)) = d(p,p'')$$
   and from $q \in \tilde Q$ we know that $\pi^{-1}(q)$ has a unique point $\ell_p(q) =p'=p''$ in optimal position to $p$, hence, indeed
   $$ \ell_{p_n}(q) \to \ell_p(q)\,.$$
\end{proof}

\begin{Th}\label{th:F=Mstar} $F= M^*$.\end{Th}

\begin{proof}
 In order to see $F \subseteq M^*$ we show $p\in M^0 \Rightarrow p \not \in F$. To this end, let $p \in M^0$ and set $q=\pi(p)$. Then there is $e\neq g \in G$ such that $g.p=p$. Using the notation from Section \ref{scn:smooth}, since $M^*$ is open and dense, there is an open set $U \subset M^*\cap \big(M\setminus \Cut(\pi^{-1}(q))\big)$. Then for every $p' \in U$ there is $g'\in G$ such that the first component of $\Phi^{-1}_q(p')$ ($\Phi_q$ maps from the tangent bundle into the manifold, cf. (\ref{eq:Phiq}) is $(g')^{-1}.p$, i.e. $g'.p'$ is in optimal position to $p$. However, also $gg'.p' \neq g'.p'$ (since $p'\in M^*$) is in optimal position to $p$, yielding $U\cap F_p = \emptyset$. Since $\vol(U) > 0$, thus $p\not\in F$.
 
 $M^* \subseteq F$: Let $p \in M^*$ with $\pi(p) =:q \in Q^*$. If $p' \in M\setminus F_p$ is in optimal position to $p$ then there is also $p'\neq p'' \in \pi^{-1}(q')$,  with $q' := \pi(p') \in Q$, in optimal position to $p$ and hence $v',v'' \in H_pM$ with $v'\neq v''$ such that $p' = \exp_p v'$ and $p'' = \exp_p v''$ as shortest geodesics between points in optimal position are horizontal \citep[p 10/11]{HHM07}. In consequence, since $p\in M^*$, these project to  two different geodesics 
 in $Q$ from $q$ to $q'$. Hence either $q'\in Q^0$, i.e. $p' \in M^0$ which has zero volume (due to \citet[Proposition 2.7.1]{duistermaat2012lie} every tube in $M$ meets only finitely many singular orbits, and $M$, as a manifold, has a countable atlas) or $q' \in \Cut(q)$ (with respect to the manifold $Q^*$) which, as $\vol_\sharp$ is the Riemannian volume for $Q^*$, has $\vol_\sharp(\Cut(q))=0$ (see \citet[p. 123]{chavel1995riemannian} and \citet{ItohTanaka2001}). In consequence $p' \in \pi^{-1}(\Cut(q))$ with $\vol\big(\pi^{-1}(\Cut(q))\big) = \vol_\sharp(\Cut(q))=0$. 
 In summary $\vol(M\setminus F_p) =0$ yielding $p\in F$.
\end{proof}

In consequence of the above Theorem, we have the following extension of (vi) of Lemma \ref{lem:basic-lifts}.

\begin{Cor}\label{cor:improve-lem-basic-lift}
     Let $X$ be absolutely continuous w.r.t. $\vol$ and with unique Fr\'echet mean $\{\nu\} = \argmin_{q\in Q}  F^{\pi\circ X}(q)$. If $\Prb\{X \in Q^*\} >0$ and if $\mu \in \pi^{-1}(\nu)$ with optimal lift $\ell_\mu$ through $\mu$, then
    $$\{\mu\} = \argmin_{p\in M}  F^{\ell_\mu \circ \pi\circ X}(p)\,.$$
\end{Cor}

\begin{proof}
    From $\Prb\{X \in Q^*\} >0$ infer with (v) of Lemma \ref{lem:basic-lifts} that $\nu \in Q^*$, i.e $\mu \in M^*=F$ for every $\mu \in \pi^{-1}(\nu)$, due to Theorem \ref{th:F=Mstar}. By definition of $F$, then $\vol\{p \in M: p \not \in L'_\mu\} =0$ and thus $\Prb\{X  \in L'_\mu\} >0$ by absolute continuity of $X$ w.r.t. $\vol$. The assertion now follows from (vi) of Lemma \ref{lem:basic-lifts}.
\end{proof}



\section{Almost Everywhere Continuity and Uniqueness of Optimal Lifts}

\begin{Th}\label{th:cont-and-uniq-opt-lift} Let $p \in M^*$. Then there is a Borel measurable $Q_p\subset Q$ such that
 \begin{enumerate}[label=(\roman*)]
  \item $\vol\big(M\setminus \pi^{-1}(Q_p)\big) =0$,
  \item $\ell_p = \ell'_p$ on $Q_p$ for any two optimal lifts $\ell_p$ and $\ell'_p$ through $p$,
  \item $Q_p \to M,\quad q\mapsto \ell_p(q)$ is continuous. 
 \end{enumerate}
\end{Th}

\begin{proof}
 Setting $Q_p = \pi(F_p)$ with $\pi^{-1}(Q_p) = F_p$ (due to (ii) Lemma \ref{lem:F_p_basics}), this is a consequence of Lemma \ref{lem:cont-and-uniq-opt-lift} below, since for $p\in F$, $0=\vol(M\setminus F_p)$ by definition and $F=M^*$ by Theorem \ref{th:F=Mstar}.
\end{proof}

\begin{Lem}\label{lem:cont-and-uniq-opt-lift}
Let $p \in M$ and $\pi(F_p) \ni q'_n \to q'\in \pi(F_p)$ with $p'_n \in L_p\cap \pi^{-1}(q'_n)$ and $p'\in L_p\cap \pi^{-1}(q')$. Then $p'_n \to p'$.
\end{Lem}

\begin{proof}
 By definition of $F_p$, both $p'_n$ and $p'$ are uniquely in optimal position to $p$. With $q = \pi^{-1}(p)$ we have thus that
 $$d(p'_n,p) = d_Q(q'_n,q) \to d_Q(q',q)$$
 yielding that $p'_n$ is bounded and hence it has a cluster point $p''\in M$ which is in $\pi^{-1}(q')$ since the latter is closed. Hence $d(p'',p)$ agrees with the limit $d_Q(q',q)$ of $d(p'_n,p)$, i.e. $p''$ is in optimal position to $p$. By uniqueness conclude $p''=p'$ yielding the assertion. 
\end{proof}

\section{Two-Sample Tests Using Optimal Lifts}\label{scn:tests}

Here we are concerned with generalizations of the classical two-sample test for 
$$ X_1,\ldots, X_n\iid X\mbox{ and } Y_1,\ldots, Y_m\iid Y\,, $$
jointly independent on $M= \RR^d$, $d\in \NN$, with population and sample means 
$$ \bar X_n = \frac{1}{n}\sum_{j=1}^n X_j,\quad \EE[X],\quad\bar Y_m = \frac{1}{m}\sum_{j=1}^m X_j,\quad \EE[Y]\,,$$
respectively, covariances 
\begin{eqnarray*}
\cov_n[X] = \frac{1}{n}\sum_{j=1}^n (X_j-\bar X_n)(X_j-\bar X_n)^T,&& \cov[X] = \EE[(X-\EE[X])(X-\EE[X])^T],
\\  \cov_m[Y] = \frac{1}{m}\sum_{j=1}^m (Y_j-\bar Y_m)(Y_j-\bar Y_m)^T,&& \cov[Y] = \EE[(Y-\EE[Y])(Y-\EE[Y])^T]
\end{eqnarray*}
and \emph{pooled covariance}
$$ \cov_{n,m}[X,Y]  = \frac{1}{n+m-2}\left(n\cov_n[X] + m\cov_m[Y] \right)\,.$$
Then, under $H_0: X\sim Y \sim \cN(\EE[X],\cov[X])$  (multivariate normal distribution), 
\begin{eqnarray}\label{eq:t-stats}
\hat T&:=&  \frac{mn}{m+n} (\bar X_n - \bar Y_m)^T \left(\cov_{n,m}[X,Y]\right)^{-1} (\bar X_n - \bar Y_m)~\sim~ T_{d,m+n-2}^2 \end{eqnarray}
where $T_{d,k}^2$ denotes the well known Hotelling $T^2$-distribution with $(d,k)$ degrees of freedom, see, e.g.  \cite{mardia2024multivariate}, if $\cov_{n,m}[X,Y]$ is of rank $d$. Here we have utilized the expected value $\EE[f(X)] = \int_{\RR^d} f(x)\,d\Prb^X(x)$ for $\Prb^X$-integrable functions.

Further in case of arbitrary $X,Y$ which are not necessarily normally distributed but feature second moments $\EE[X^TX], \EE[Y^TY] < \infty$ with full rank covariances, \emph{robustness under nonnormality holds} \citep[p. 207, 321]{RomanoLehmann2005}, i.e.
$$ \hat T \leadsto Z\mbox{ by which we mean } \frac{\hat T }{Z} \inD 1$$
if $Z \sim T^2_{d,m+n-2}$ is independent and $n,m\to \infty$, under 
$$H_0: \EE[X] = \EE[Y] \mbox{ and } \Big( \frac{n}{m} \to 1\mbox{ or } \cov[X] = \cov[Y]\Big)\,.$$

With the $\alpha$-quantile $T_{d,m+n-2,\alpha}$ of the Hotelling  $T^2$-distribution with $(d,m+n-2)$ degrees of freedom defined by $\Prb\{Z\leq T_{n,m,\alpha}\} = \alpha$ the Hotelling test rejects $H_0$ at nominal level $\alpha \in [0,1]$ if
$$ \hat T > T_{d,m+n-2,1-\alpha}\,.$$
Under the above assumptions, this test has asymptotically the true level $\alpha$ \citep{RomanoLehmann2005}.

In the following we introduce four different tests for data on a quotient space $Q$. For shape spaces (Section \ref{scn:shape} introduces some classical shape spaces and a new one), the first one gives, among others, the classical two-sample test based on \emph{Procrustes analysis} \citep{DM16} or for \emph{intrinsic MANOVA} \citep{HHM09}. The other two are new and build on the strong law for optimal lifts (Theorem \ref{th:SL-optimal-lift}) and their local smoothness (Corollary \ref{cor:lift-smooth}). For all three tests we have the following setup.

\begin{As} We assume that
$$ W_1,\ldots, W_n\iid W\mbox{ and } Z_1,\ldots, Z_m\iid Z\,, $$
are jointly independent on $Q$, $\dim(Q^*)=d$ with population Fr\'echet means $\nu^W$ and $\nu^Z$, respectively, and measurable selections $\nu_n^W$ and $\nu_m^Z$ of sample Fr\'echet means,  respectively. Further, let $\mu_W \in \pi^{-1}(\nu_W)$  with optimal lift $\ell_{\mu_W}$ through $\mu_W$ and $\mu_Z \in \pi^{-1}(\nu_Z)$ with optimal lift $\ell_{\mu_Z}$ through $\mu_Z$. We assume that $W,Z$ are absolutely continuously distributed with respect to $\pi \circ \vol$.

Moreover, we assume that either $n/m \to 1$ or that $\cov[X] = \cov[\theta_{\mu_W,\mu_Z}\circ Y]$ where 
\begin{eqnarray*}X&: =& (\exp_{\mu_W})^{-1} \circ \ell_{\mu_W}\circ W,\\ 
Y &:=& (\exp_{\mu_Z})^{-1} \circ \ell_{\mu_Z}\circ Z
\end{eqnarray*}
 and $\theta_{\mu_W,\mu_Z}$ is the parallel transport along a length minimizing geodesic from $\mu_Z$ to $\mu_W$. 
\end{As}

Under these assumptions,
\begin{enumerate}
 \item by Theorem \ref{th:Le-Barden}, $X$ and $Y$ are a.s. well defined, 
 \item $\nu^W , \nu^Z \in Q^*$ by the manifold stability Theorem \ref{th:manifold-stability},
\end{enumerate}
and in the following, we test for 
$$H_0: \nu^W =  \nu^Z$$
at a given level $\alpha \in [0,1]$.

\begin{Test}[Pooled Lifting]\label{test:pooled-lifting} Let $\nu^{W,Z}_{n,m}$ be a measurable selection of the pooled sample Fr\'echet means of $ W_1,\ldots, W_n, Z_1,\ldots, Z_m$, let $\mu^{W,Z}_{n,m}\in \pi^{-1}(\nu^{W,Z}_{n,m})$ be a measurable selection and let $\ell_{\mu^{W,Z}_{n,m}}$ be an optimal lift through $\mu^{W,Z}_{n,m}$. Setting
\begin{eqnarray*}
X_j&: =& (\exp_{\mu^{W,Z}_{n,m}})^{-1} \circ \ell_{\mu^{W,Z}_{n,m}}\circ W_j, \quad 1\leq j \leq n,\\ 
Y_j &:=& (\exp_{\mu^{W,Z}_{n,m}})^{-1} \circ \ell_{\mu^{W,Z}_{n,m}}\circ Z_j, \quad 1\leq j \leq m
\end{eqnarray*}
and $T_{pooled} := \hat T$ from (\ref{eq:t-stats}), reject $H_0$ if $T_{pooled} > T_{n,m,1-\alpha}$.
 
\end{Test}

In this test the sample covariance matrices have been computed using averages in the tangent spaces. Alternatively, these averages can be replaced in an intrinsic fashion by preimages of sample Fr\'echet means. 

\begin{Test}[Pooled Lifting Intrinsically]\label{test:pooled-intrinsic-lifting}  With the notation from Test  \ref{test:pooled-lifting}, setting 
\begin{eqnarray*}
\bar X^{(i)}_n := (\exp_{\mu^{W,Z}_{n,m}})^{-1} \circ \ell_{\mu^{W,Z}_{n,m}}\circ \nu_n^W,&& 
\bar Y^{(i)}_m := (\exp_{\mu^{W,Z}_{n,m}})^{-1} \circ \ell_{\mu^{W,Z}_{n,m}}\circ \nu_m^Z
\end{eqnarray*}
define the \emph{intrinsic sample covariances} and corresponding intrinsic test statistic
\begin{eqnarray}\label{eq:t-stats-i} \nonumber
\cov^{(i)}_n[X] &:=& \frac{1}{n}\sum_{j=1}^n (X_j- \bar X^{(i)}_n)(X_j-\bar X^{(i)}_n)^T,\\  \nonumber
\cov^{(i)}_m[Y] &:=& \frac{1}{m}\sum_{j=1}^m (Y_j-\bar Y^{(i)}_m)(Y_j-\bar Y^{(i)}_m)^T,\\\nonumber
\cov^{(i)}_{n,m}[X,Y] &:=& \frac{1}{n+m-2}\left(n\cov^{(i)}_n[X] + m\cov^{(i)}_m[Y] \right),\\
\hat T^{(i)}&:=&  \frac{mn}{m+n} (\bar X^{(i)}_n - \bar Y^{(i)}_m)^T \left(\cov^{(i)}_{n,m}[X,Y]\right)^{-1} (\bar X^{(i)}_n - \bar Y^{(i)}_m) 
\end{eqnarray}
Then, setting $T_{pooledI} := \hat T^{(i)}$ from (\ref{eq:t-stats-i}), reject $H_0$ if $T_{pooledI} > T_{d,m+n-2,1-\alpha}$.
\end{Test}

Alternatively we can lift each sample separately.

\begin{Test}[Individual Lifting]\label{test:individual-lifting} With the notation from Test  \ref{test:pooled-lifting} let $\mu^{W}_{n} \in \pi^{-1}(\nu^{W}_{n} )\cap L_{\mu^{W,Z}_{n,m}}$, $\mu^{Z}_{n} \in \pi^{-1}(\nu^{Z}_{n} )\cap L_{\mu^{W,Z}_{n,m}}$, i.e. both are a.s. uniquely (due to Theorems \ref{th:manifold-stability} and \ref{th:F=Mstar}) in optimal position to $\mu^{W,Z}_{n,m}$ and let $\ell_{\mu^{W}_{n}}$ be an optimal lift through $\mu^{W}_{n}$, $\ell_{\mu^{Z}_{m}}$ be an optimal lift through $\mu^{Z}_{m}$. Then setting
\begin{eqnarray*}
\bar X^{(i)}_n := (\exp_{\mu^{W}_{n}})^{-1} \circ \ell_{\mu^{W}_{n}}\circ \nu_n^W,&& 
\bar Y^{(i)}_m := (\exp_{\mu^{Z}_{m}})^{-1} \circ \ell_{\mu^{Z}_{m}}\circ \nu_m^Z,
\end{eqnarray*}
as well as,
\begin{eqnarray*}
X_j&: =& (\exp_{\mu^{W}_{n}})^{-1} \circ \ell_{\mu^{W}_{n}}\circ W_j, \quad 1\leq j \leq n,\\ 
Y_j &:=& (\exp_{\mu^{Z}_{m}})^{-1} \circ \ell_{\mu^{Z}_{m}}\circ Z_j, \quad 1\leq j \leq m
\end{eqnarray*}
and $T_{individual} := \hat T^{(i)}$ from (\ref{eq:t-stats-i}), reject $H_0$ if $T_{individual} > T_{d,n+m-2,1-\alpha}$.
 
\end{Test}

If it is computationally too expensive to compute the pooled sample mean, in particular in view of the bootstrap tests introduced further below, instead optimal position with respect to the other sample's mean.

\begin{Test}[Individual Asymmetric Lifting]\label{test:individual-asymetric-lifting} Measurably select $\mu^{W}_{n} \in \pi^{-1}(\nu^{W}_{n} )$ and $\mu^{Z}_{n} \in \pi^{-1}(\nu^{Z}_{n} )\cap L_{\mu^{W}_{n}}$, i.e. both are a.s. uniquely (due to Theorems \ref{th:manifold-stability} and \ref{th:F=Mstar}) in optimal position to one another and define all other quantities as in Test \ref{test:individual-lifting}. With those compute $T_{individualA} := \hat T^{(i)}$ from (\ref{eq:t-stats}). Then reject $H_0$ if $T_{individualA} > T_{d,n+m-2,,1-\alpha}$.
 
\end{Test}

\begin{Rm}\label{rm:test-smeary}
 As detailed above and in \cite{HE_Handbook_2020}, if the Hessian of the Fr\'echet function  $q\mapsto F^{\ell_{p'}\circ W}(q)$ in the exponential chart has all eigenvalues equal to $2$ (as is the case if $Q$ is Euclidean) under $H_0$ at $q = \mu^W = \mu^Z$ the various tests have asymptotic level $\alpha$ for the various choices of $p'= \nu^W \in \pi^{-1}(\mu^W)$. The smaller the eigenvalues, the more drastic the deviation \citep{hundrieser2024lowerA}, and in case of zeroes, a phenomenon called (directional) \emph{smeariness} kicks in \citep{EltznereHuckemann2019,DHH-GSI21}, making the tests completely inapplicable. 
\end{Rm}

For this reason, \cite{HuElHu2024} propose to use bootstrap versions of the above quantile based tests, namely simulating quantiles via suitably resampling, as follows. Here is the generic procedure from \cite{eltzner2017bootstrapping} applicable to each of the three latter tests. Taking averages in tangent spaces instead of intrinsic sample means, the thus modified bootstrap version is also applicable to the first Test \ref{test:pooled-lifting}. In fact, the following test, inspired by the test from \citet{welch1947generalization}, also works well in case of different covariances.

\begin{Test}[Bootstrap Version for Tests \ref{test:pooled-intrinsic-lifting}, \ref{test:individual-lifting} and \ref{test:individual-asymetric-lifting}]\label{test:bootstrap}
In a first round of bootstrapping from $W_1,\ldots,W_n$ obtain $\cov^{*,(i)}_n[W]$ and of bootstrapping from $Z_1,\ldots,Z_m$ obtain $\cov^{*,(i)}_m[Z]$. Then set 
$$ C : = \cov^{*,(i)}_n[W] + \cov^{*,(i)}_m[Z]\,.$$ 
In a second independent round of bootstrapping $B\in \RR$ (large) times from $W_1,\ldots,W_n$ and from $Z_1,\ldots,Z_m$ obtain 
$$ d^{*,X,b}:= \bar X^{*,b}_n - \bar X^{(i)}_n,\quad d^{*,Y,b}:= \bar Y^{*,b}_m - \bar X^{(i)}_m,\quad b=1,\ldots,B\,.$$
From these obtain the bootstrap $\alpha$-quantile $T^*_{m,n,\alpha}$ of the data 
$$\left(d^{*,X,b}- d^{*,Y,b}\right)^T C^{-1}\left(d^{*,X,b}- d^{*,Y,b}\right)\,.$$
Then, reject $H_0$ if 
$$\hat T^{(i,*)} :=  (\bar X^{(i)}_n - \bar Y^{(i)}_m)^T C^{-1}  (\bar X^{(i)}_n - \bar Y^{(i)}_m) > T^*_{d,m+n-2,\alpha}\,.$$
\end{Test}

For data on the circle \cite{HuElHu2024} have shown that the corresponding tests keep the level $\alpha$ asymptotically. For general (quotient) spaces this is open, but simulations indicate that the levels are asymptotically kept. We illustrate some simulations in Section \ref{scn:simulation-application}.

We note that \cite{BP05} have been the first to introduce intrinsic two-sample tests for manifold data. One of their test is based on quantiles of a $\chi^2$-distribution and the other is a bootstrap alternative (see also \cite{BL17}).

\section{Reverse Labeling Reflection Shape Spaces}\label{scn:shape}

\paragraph{Classical shape spaces and reflection shape spaces}
\citep{DM16} model, modulo certain group actions, landmark configurations $A=(a_1,\ldots,a_k)\in \RR^{m\times k}$, giving the position of $3\leq k \in \NN$ landmark vectors $a_1,\ldots,a_k \in \RR^m$ of dimension $m\in \NN$, where we assume that $2\leq m < k$. One basic group action is that of translations and rotations (Euclidean motions), given by
$$ SO(m) \sdp \RR^m \mbox{ acting on }M=\RR^{m\times k}\mbox{ via } A\stackrel{(g,a)}{\to} gA + a1_k^T \quad (g,a) \in SO(m) \sdp \RR^m\,.$$ 
Here, with the standard unit vectors $e_1,\ldots,e_k \in \RR^m$, $1_k := e_1+\ldots + e_k\in \RR^k$ denotes the vector with all entries equal to one and the (semidirect) product of $(g,a), (h,b) \in   SO(m) \sdp \RR^m $ is given by $(gh,b+ga)$. Letting $I_k =(e_1,\ldots,e_k) \in \RR^{k\times k}$ denote the unit matrix, there are two popular ways to filter out translation: 
\begin{enumerate}[label=(\roman*)]
 \item Centering by subtracting the mean landmark: 
 $\RR^{m\times k} \to \cC_m^k$, $A\mapsto A \big(I_k - \frac{1}{k}1_k1_k^T\big)$, here
 $$\cC_m^k :=\{B\in \RR^{m\times k}: B 1_k=0\}$$
 denotes the \emph{centered configurations}.
 \item Helmertizing: $\RR^{m\times k} \to \RR^{m\times (k-1)}, A\mapsto AH_k$ where $H_k \in \RR^{k\times (k-1)}$ is chosen such that $(H_k|\frac{1}{\sqrt{k}}1_k) \in SO(k)$. Usually $H_k$ is chosen as the \emph{sub Helmert matrix}:
 $$ H_k =\left(\begin{array}{cccc}
	   {1 \over \sqrt{2}} &{1 \over \sqrt{6}} & \dots& {1 \over
	   \sqrt{k(k-1)}}\\ 
	   -{1 \over \sqrt{2}} &{1 \over \sqrt{6}} & \dots& {1 \over \sqrt{k(k-1)}} \\ 
	   0&-{2 \over \sqrt{6}} &\dots& {1 \over \sqrt{k(k-1)}} \\ 
	   \vdots&\vdots&\ddots&\vdots \\
	   0&0 &\dots& -{k-1 \over \sqrt{k(k-1)}}\\ 
		     \end{array}
	 \right)\,.$$
\end{enumerate}
As usual, $\RR^{m\times k}$ is equipped with the standard Euclidean inner product
$$ \langle A,B \rangle  := \tr(A^TB),\quad A,B\in \RR^{m\times k}$$
yielding the Frobenius norm
$$ \|A\| := \sqrt{\tr(A^TA)}$$ 
where $\tr$ denotes the matrix trace. Since the action of scaling 
$$ \RR_+\mbox{ acting on } \RR^{m\times k} \mbox{ via } A \stackrel{\lambda}{\to} \lambda A,\quad \lambda \in \RR_+$$ 
is neither isometric nor proper (the only closed orbit is that of $0\in \RR^{m\times k}$), representatives on the unit hypersphere 
      $$\SSS^{m\times k-1} := \{B\in \RR^{m\times l}: \|B\| = 1\} $$ 
are chosen. Thus, taking away the problematic diagonal orbit originating from single point configurations,
$$\cD_m^k :=\{A=(a,\ldots,a)\in \RR^{m\times k}: a \in \RR^m\}\,,$$ 
filtering out translation and scaling, one arrives at 
$$\RR^{m\times k} \setminus  \cD_m^k  \to \left\{\begin{array}{rl}
                                                 \cC_m^k \setminus \{0\}&\mbox{ centering}\\
                                                 \RR^{m\times (k-1)}  \setminus \{0\}&\mbox{ Helmertizing}
                                                \end{array}\right\}\stackrel{B \mapsto \frac{B}{\|B\|}}{\to}\left\{\begin{array}{rl}
                                                 \SSS^{m\times k-1}\cap \cC_m^k&\mbox{ centering}\\
                                                 \SSS^{m\times (k-1)-1} &\mbox{ Helmertizing}
                                                \end{array}\right.\,.$$
   Usually, $\SSS^{m\times (k-1)-1}$ is called the \emph{pre-shape sphere} and we also call 
   $$ \cS_m^k := \SSS^{m\times k-1}\cap \cC_m^k$$
   the \emph{centered pre-shape sphere}.
   
   In the last step, rotations are filtered out leading to the classical \emph{shape space} 
   $$ \Sigma_m^k\mbox{ viewed as either } \cS_m^k/SO(m) \mbox{ or } \SSS^{m\times (k-1)-1}/SO(m)\,,$$
   or orthogonal transformations are filtered out leading to the classical \emph{reflection shape space} 
   $$ \cR\Sigma_m^k\mbox{ viewed as either } \cS_m^k/O(m) \mbox{ or } \SSS^{m\times (k-1)-1}/O(m)\,. $$
   Here $SO(m)$ and $O(m)$ act, as above, via multiplication from the left.

   Notably, for $A,B \in \SSS^{m\times (k-1)-1}$ their spherical distance is
   $d(A,B)  = \arccos \tr(A^TB)$
   and hence,
   $$\min_{R \in O(m)} d(A,RB) = \arccos \max_{R \in O(m)} \tr(A^TRB) = \arccos \tr\Lambda$$
   where $BA^T = U\Lambda V^T$ is a singular value decomposition. In consequence the maximal possible distance in $\cR \Sigma_m^k$ is $\pi/2$ and a similar argument yields the same maximal distance in $\Sigma_m^k$.
   
   \paragraph{Planar shape and planar reflection shape} (see additionally \cite{HZ06,HH09}) take advantage from identifying $A= (a_1,\ldots,a_k) \in \RR^{2\times k}$ with the complex row vector $z=(z_1,\ldots,z_k)\in \CC^k$, where $a_j = (\RE(z_j), \IM(z_j))^T$, $j=1,\ldots,k$. Rotation by the angle $\phi \in [0,2\pi)$ corresponds to complex multiplication from the left by $e^{i\phi}$. In consequence, $ \Sigma_2^k$ carries a canonical Riemannian manifold structure of the complex projective space $\CC P^{k-2}$ of complex dimension $k-2$ (real dimension $2k-4$). 
   
   Reflection then corresponds to joint complex conjugation:
   $$R.z = R.(z_1,\ldots,z_k) := (\bar z_1,\ldots,\bar z_k) = \bar z$$
   and hence
   $\cR\Sigma_2^k$ also has a singular stratum comprising all shapes
   $$ [z] = \{e^{i\phi}z: 0\leq \phi < 2\pi\}\mbox{ where } z \in \{(z_1,\ldots,z_k) \in \SSS^{m\times k-1}\cap \cC_m^k: z_j \in \RR\mbox{ for all }1\leq j \leq k\}\,.$$
   In particular $\Sigma_2^3$ is a sphere with radius $1/2$ and $\cR\Sigma_2^3$ can be viewed as its closed upper hemisphere. This is due to the explicit Hopf fibration (the Riemannian isometry is achieved by halving the radius $\SSS^2$):
   $$\SSS^{3} \to \SSS^{3}/\SSS^1 \cong \SSS^2,\quad (z_1,z_2) \mapsto \left(\begin{array}{c}
                         2\RE(z_1 \bar z_2)\\
                         2\IM(z_1 \bar z_2)\\
                         |z_1|^2-|z_2|^2
                        \end{array}\right) = \left(\begin{array}{c}
                         2\RE(z_1 \bar z_2)\\
                         2\IM(z_1 \bar z_2)\\
                         2|z_1|^2-1
                        \end{array}\right)\,. $$ 
   
   Further, for $z,w \in \SSS^{2\times k-1}\subset \CC^{k}$ with $z1_k = 0=w1 _k$ we have
   \begin{eqnarray*}
   d_{\Sigma_2^k}([z]_{\Sigma_2^k},[w]_{\Sigma_2^k}) &=& \arccos |z \bar w^T| \\
   d_{\cR\Sigma_2^k}([z]_{\cR\Sigma_2^k},[w]_{\cR\Sigma_2^k}) &=& \min\{\arccos |z w^T|, \arccos |z \bar w^T|\}\,.
   \end{eqnarray*}

   \paragraph{The new reverse labeling reflection shape spaces} are required in some applications, for instance when studying the shape of planar of one-dimensional non oriented string structures using landmark configurations, a frequent task in (micro)biology. Then, their shapes are preserved when placing landmarks from one end to the other, also in reversed order. To this end,
   introduce the \emph{reverse labeling} action
   $$ \RR^{m\times k} \to \RR^{m\times k},\quad A\mapsto AL_k\mbox{ with }
   L_k:=(e_k,\ldots,e_1) = \left(\begin{array}{ccc} 0&\ldots&1\\  \vdots &\rotatebox{90}{$\ddots$}&\vdots \\ 1&\ldots&0 \end{array}\right) \in \RR^{k\times k}\,.$$
   Notably this action commutes with centering and with the action of $O(m)$. 
   Since $L_k^2 = I_k$, setting $G=\{L_k,I_k\}$, we have the following.
   
   \begin{Def} The \emph{reverse labeling reflection shape space} for $m$-dimensional $k$-landmark configurations is given by 
   $$\cR\cR\Sigma_m^k := \cR\Sigma_m^k/G:= \big(\cS_m^k/G\big)/O(m)\,.$$
    \end{Def}   
   Reverse labeling is obviously isometric but it does not commute with Helmertizing since $H_k^TL_kH_k \neq L_{k-1}$ :
   
   \begin{Lem} For $H_k\in\RR^{k\times k-1}$ with $(H_k|\frac{1}{\sqrt{k}}1_k) \in SO(k)$, $L_k$ as above and arbitrary $A \in \RR^{m\times k}$ we have 
   $$AL_kH_k  = AH_k (H_k^TL_kH_k)\,,\mbox{ where }H_k^TL_kH_k \in O(k-1)\,.$$
   \end{Lem}

   \begin{proof} The first equality follows from $H_kH_k^TL_k = (I_k - \frac{1}{k}1_k1_k^T) L_k = L_k - \frac{1}{k}1_k1_k^T$ and the fact that $1_k^TH_k=0$ by construction. For the second assertion consider
   $$(H_k^TL_kH_k)^T H_k^TL_kH_k = H_k^TL_kH_kH_k^TL_kH = H_k^TL_k(L_k - \frac{1}{k}1_k1_k^T)H_k = I_{k-1}$$
   since $L_k^2= I_k$ and $H_k^TH_k = I_{k-1}$.
   \end{proof}
   
   In consequence, for reverse labeling reflection shape we only use the centered pre-shape sphere $\cS_m^k$ as above, on which $G$ acts from the right and $O(m)$ from the left. Thus, for $A,B \in \cS_m^k$, we have  
   $$ d_{\cR\cR\Sigma_m^k}([A]_{\cR\cR\Sigma_m^k},[B]_{\cR\cR\Sigma_m^k}) = \min\{\arccos\tr\Lambda,\arccos\tr\Lambda'\}$$
   where $BA^T = U\Lambda V^T$ and $BL_kA^T = U'\Lambda' (V')^T$ are singular value decompositions.
   
   \paragraph{For planar reverse labeling reflection shape} in complex representation,  for $z=(z_1,\ldots,z_k) \in \cS_2 ^k$,  we set $\tilde z := (z_k,\ldots,z_1) = zL_k$, so that
   $$ d_{\cR\cR\Sigma_2^k}([z]_{\cR\cR\Sigma_2^k},[w]_{\cR\cR\Sigma_2^k}) = \min\{\arccos |z w^T|, \arccos |z \bar w^T|,\arccos |\tilde z  w^T|, \arccos |\tilde z \bar w^T|\}\,.$$
    
    \begin{Lem}
     $\cR\cR \Sigma_2^3$ can be viewed as a closed third of the closed hemisphere $\cR\Sigma_2^3$ in the sphere $\Sigma_2^3$. In particular, the maximal distance in $\cR\cR \Sigma_2^3$ is $\pi/6$.
    \end{Lem}

    \begin{proof}
     For $z=(z_1,z_2,z_3) \in \cS_2 ^3$ let $w=(w_1,w_2) = zH_3$. Then, using the Helmert submatrix,
     $$ (w_1,w_2) H_3^TL_3H_3 = (w_1,w_2)\frac{1}{2}\left(\begin{array}{cc} 1 &-\sqrt{3}\\-\sqrt{3} & -1\end{array}\right) = - \frac{1}{2}\left(w_2\,\sqrt{3} - w_1,w_2 +w_1\,\sqrt{3}\right)\,,$$
     and
     \begin{eqnarray*}\left(w_2\,\sqrt{3} - w_1\right)\,\left(\overline{w_2 +w_1\,\sqrt{3}}\right) &=& -\,\sqrt{3}(|w_1|^2 -|w_2|^2) -w_1\bar w_2 + 3w_2\bar w_1\\
     &=& -\,\sqrt{3}(|w_1|^2 -|w_2|^2) -2\RE(w_1\bar w_2) - i 4\IM(w_1\bar w_2)\,,\\
     \left|w_2\,\sqrt{3} - w_1\right|^2-\left|\overline{w_2 +w_1\,\sqrt{3}}\right|^2 &=& 2(|w_1|^2 - |w_2|^2) - 4\sqrt{3}\RE(w_1\bar w_2)\,.
      \end{eqnarray*}
     Hence, under reverse labeling,
     $$\left(\begin{array}{c}x_1\\x_2\\x_3\end{array}\right) =  \left(\begin{array}{c}
                         2\RE(w_1 \bar w_2)\\
                         2\IM(w_1 \bar w_2)\\
                         |w_1|^2-|w_2|^2
                        \end{array}\right) $$ 
     is mapped to 
     $$\left(\begin{array}{c}
                         \frac{1}{2}\RE\left(\left(w_2\,\sqrt{3} - w_1\right)\,\left(\overline{w_2 +w_1\,\sqrt{3}}\right)\right)\\
                         \frac{1}{2}\IM\left(\left(w_2\,\sqrt{3} - w_1\right)\,\left(\overline{w_2 +w_1\,\sqrt{3}}\right)\right)\\
                         \frac{1}{4}\left(\left|w_2\,\sqrt{3} - w_1\right|^2-\left|\overline{w_2 +w_1\,\sqrt{3}}\right|^2\right)
                        \end{array}\right) 
                        =-\, \frac{1}{2}\left(\begin{array}{c}x_1 + \sqrt{3}\,x_3 \\ 2 x_2 \\\sqrt{3}\,x_1 - x_3\end{array}\right)\,,   
                        $$
    i.e. accounting for a rotation by $\pi/3$ in the $(x_1,x_3)$-plane. Since reflection is filtered out by assuming $x_2\geq 0$ and taking into account that in order to achieve an isometry the radius of the unit sphere needs to be halved, the maximal distance in $\cR\cR\Sigma_2^3$ is indeed $\frac{1}{2}\,\frac{\pi}{3}$.
    \end{proof}
    
    \begin{Rm} We anticipate that the diameter of $\cR\cR\Sigma_2^k$ increases with $k$. \end{Rm}

\section{Simulations and Application}\label{scn:simulation-application}

\paragraph{Simulations on $\Sigma_2^5$.} Here we consider two different planar configurations of five landmarks, of same shape for $H_0$ and of shape distance $0.06$ for an alternative $H_1$, each with random independent identical multinormal isotropic noise on every landmark with standard deviation of $0.2$. In Table \ref{tab:perfomance_shape_space} we compare the various two-samples tests at nominal level $\alpha = 0.05$ for various sample sizes with one another and record the empirical level under $H_0$ and the power under $H_1$; we also include the bootstrap test  $T_J$ proposed by \cite{preston2010two}, originally for three-dimensional shape data. In terms of keeping the level and maximizing power, from inspection of Table \ref{tab:perfomance_shape_space}, the tests based on individual lifting tend to outperform the others tests. 
 
\begin{table}[!h]
\begin{center}
	\includegraphics[scale = 0.4]{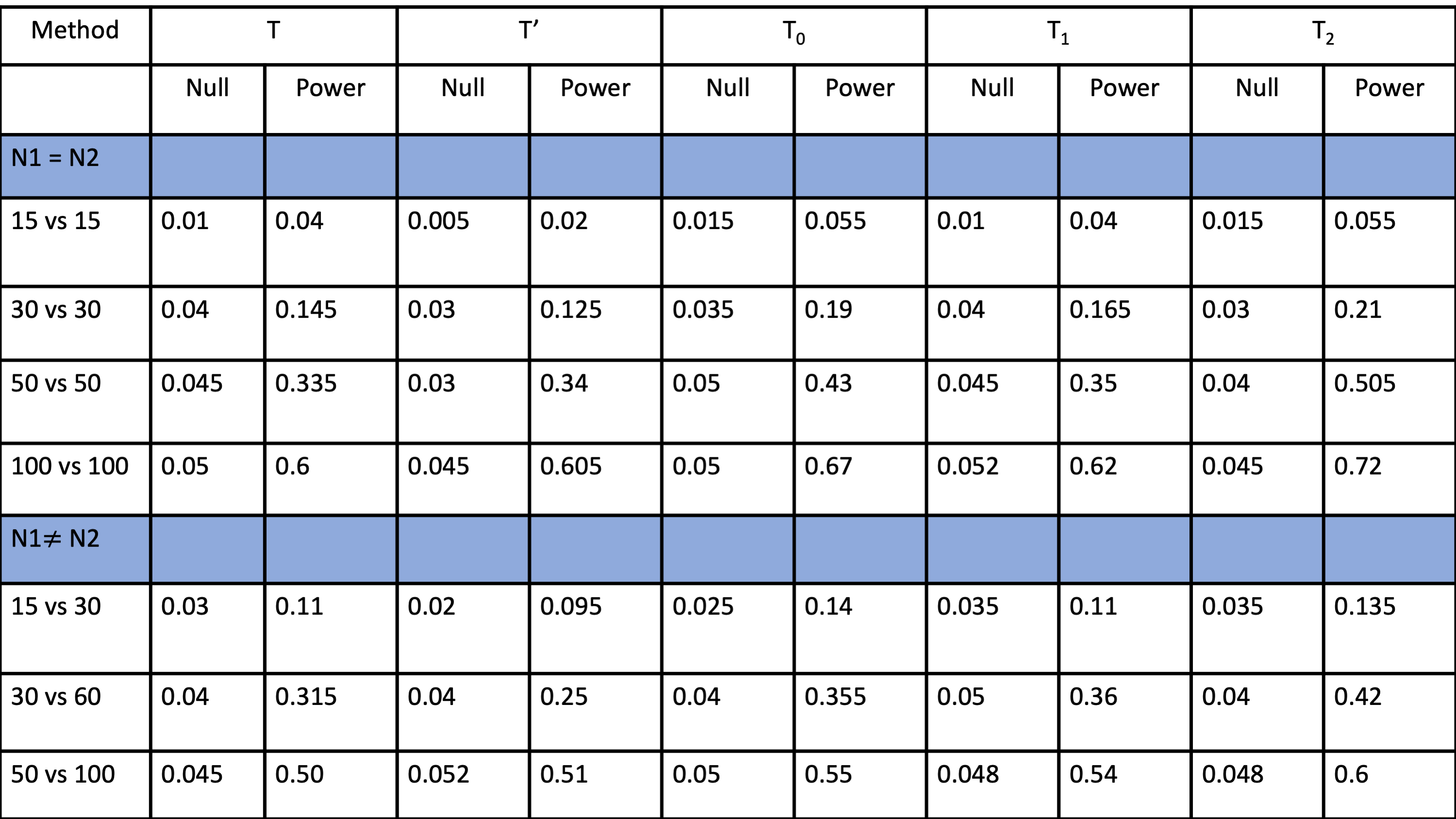}
\end{center}
\caption{\it Comparing true levels (columns labeled Null) at nominal level $0.05$ and powers (columns labeled Power) for the bootstrap versions of the various tests from Section \ref{scn:tests}: $T$ stands for ``pooled lifting'', $T'$ for $T_J$ from \cite{preston2010two},  $T_0$ for ``pooled lifting intrinsically'', $T_1$ for ``individual lifting'' and $T_2$ for ``individual asymmetric lifting''.  The numbers are averages of rejections of $H_0$ of $1000$ tests each. The first column gives the two samples sizes, which are equal for rows 4 -- 7 (alleviating asymptotic normality) and different for rows 8 -- 10.}
\label{tab:perfomance_shape_space}
\end{table}

In a second simulation, we consider the same configuration data, now in reverse labeling reflection shape space $\cR\cR\Sigma_2^5$. We perform the same tests as above and now record the results in Table \ref{tab:perfomance_RR_shape_space}. While in terms of keeping the level and maximizing power, the tests based on individual lifting still tend to outperform the others tests, this effect is weaker than for shape spaces above. Figure \ref{fig:power_curve} gives more detail by depicting  the corresponding power curves. Here, the picture is clearer, the bootstrap tests ``pooled lifting'' and $T_J$ have least power, the tests ``pooled lifting individually'' and ``individual lifting'' have higher power, the test ``individual asymmetric lifting'' has largest power. Notably their true levels are rather comparable (see Table \ref{tab:perfomance_RR_shape_space}).

\begin{table}[!h]
\begin{center}
	\includegraphics[scale = 0.4]{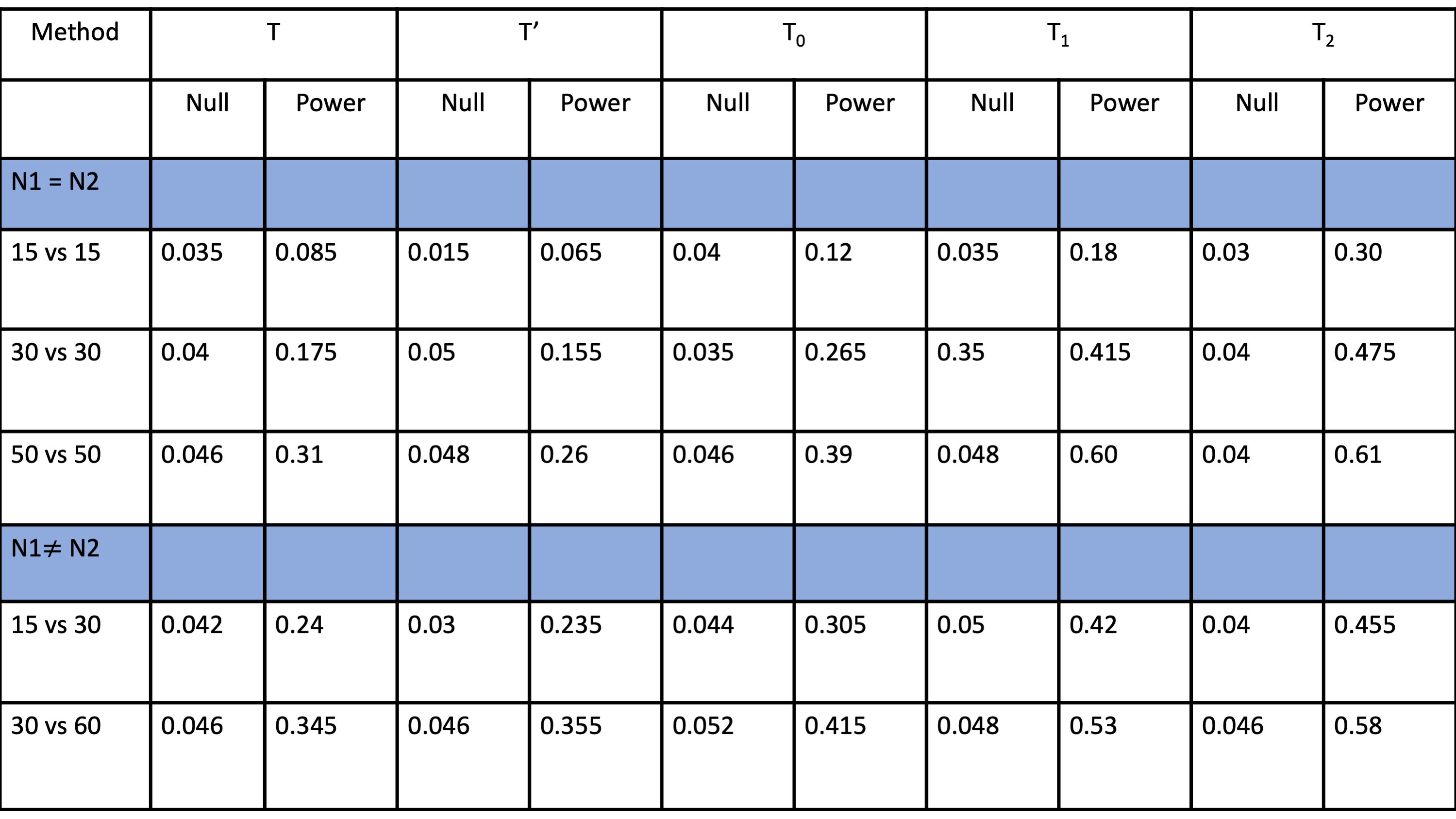}
	\caption{\it Depicting true levels and powers as in Table \ref{tab:perfomance_shape_space} for the same tests of the same data, but now for reverse labeling reflection shape in $\cR\cR\Sigma_2^5$. Again the  numbers are averages of $1000$ tests each.}
\label{tab:perfomance_RR_shape_space}
\end{center}
\end{table}

\begin{figure}[!h]
\begin{center}
	\includegraphics[scale=0.35]{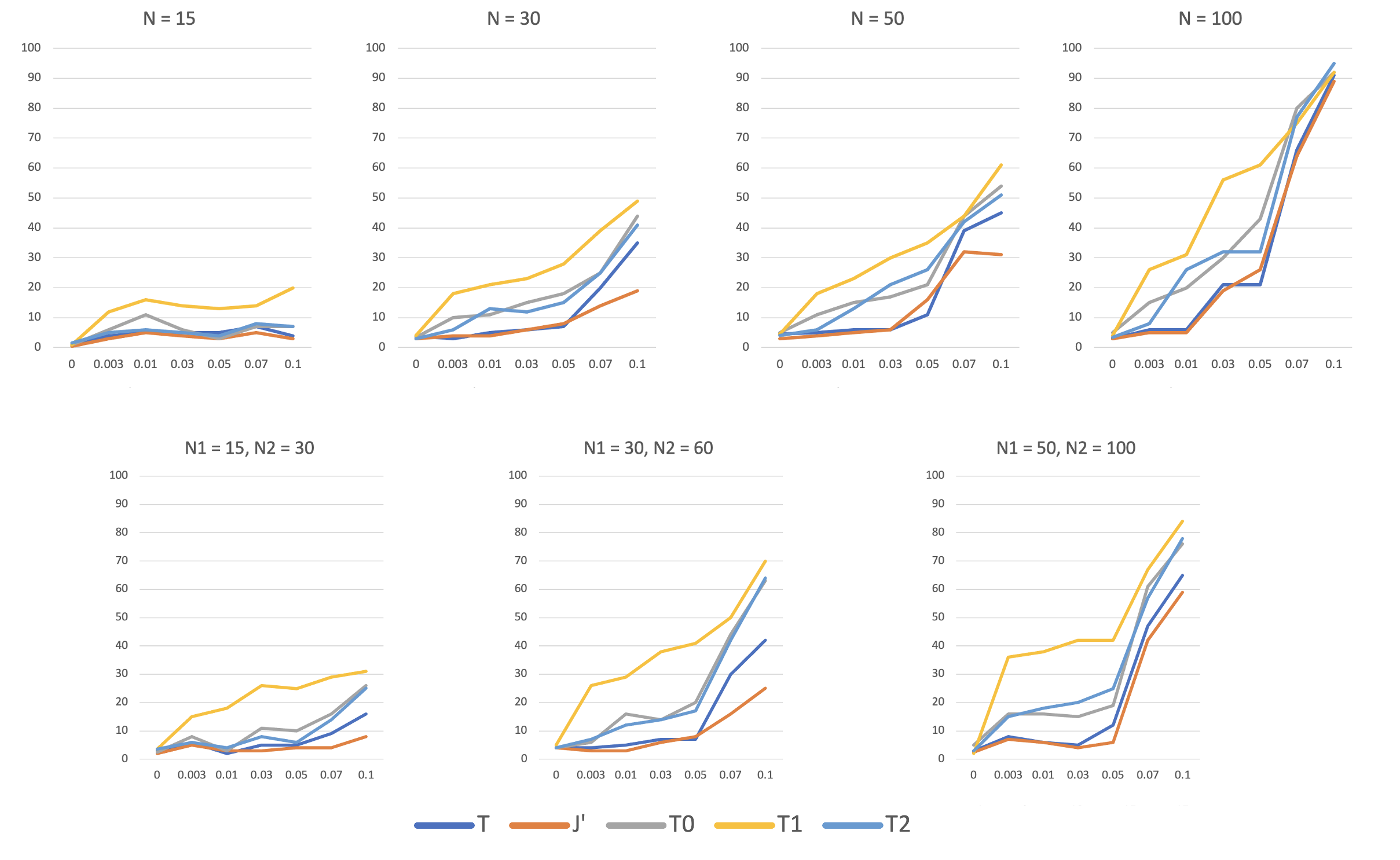}
\end{center}
\caption{\it Power curves for the data and five tests (labeling explained in Table \ref{tab:perfomance_shape_space}. The values at $0.06$ correspond to the ones recorded under power in Table \ref{tab:perfomance_RR_shape_space}. Top row for equal sample sizes (rows 4 -- 7 in Table \ref{tab:perfomance_RR_shape_space}) and bottom row for different sample sizes (rows 8 -- 10 in Table \ref{tab:perfomance_RR_shape_space}). In contrast to Table \ref{tab:perfomance_RR_shape_space} each test has been repeated $500$ times. The horizontal axis records the reverse relabeling reflection shape distance between the corresponding two population Fr\'echet means and the vertical axis depicts the percentage of rejections.}
\label{fig:power_curve}
\end{figure}

\paragraph{Application to the interaction of microtubules and intermediate filaments in biological cells.} 

In the analysis of biological cells filament structures, buckles of \emph{microtubules} play an important role, for example \citet{nolting2014mechanics}. In particular, intermediate filaments (IF) are important components which provide mechanical stability in cells and tissues. A disruption of the IF or of their connection with other cell structures cause different degenerative diseases of skin, muscle, and neurons, see, for example \citet{LZ01}. Most notably, even though tumor cells lose their normal appearance, they retain several expressions of particular IF proteins. Thus, a method to identify the IF proteins in tumor cells could help to develop the most effective treatment to destroy the tumor \citep{LZ01}. 

 \begin{figure}[h!]
\centering
\includegraphics[width=0.5\textwidth]{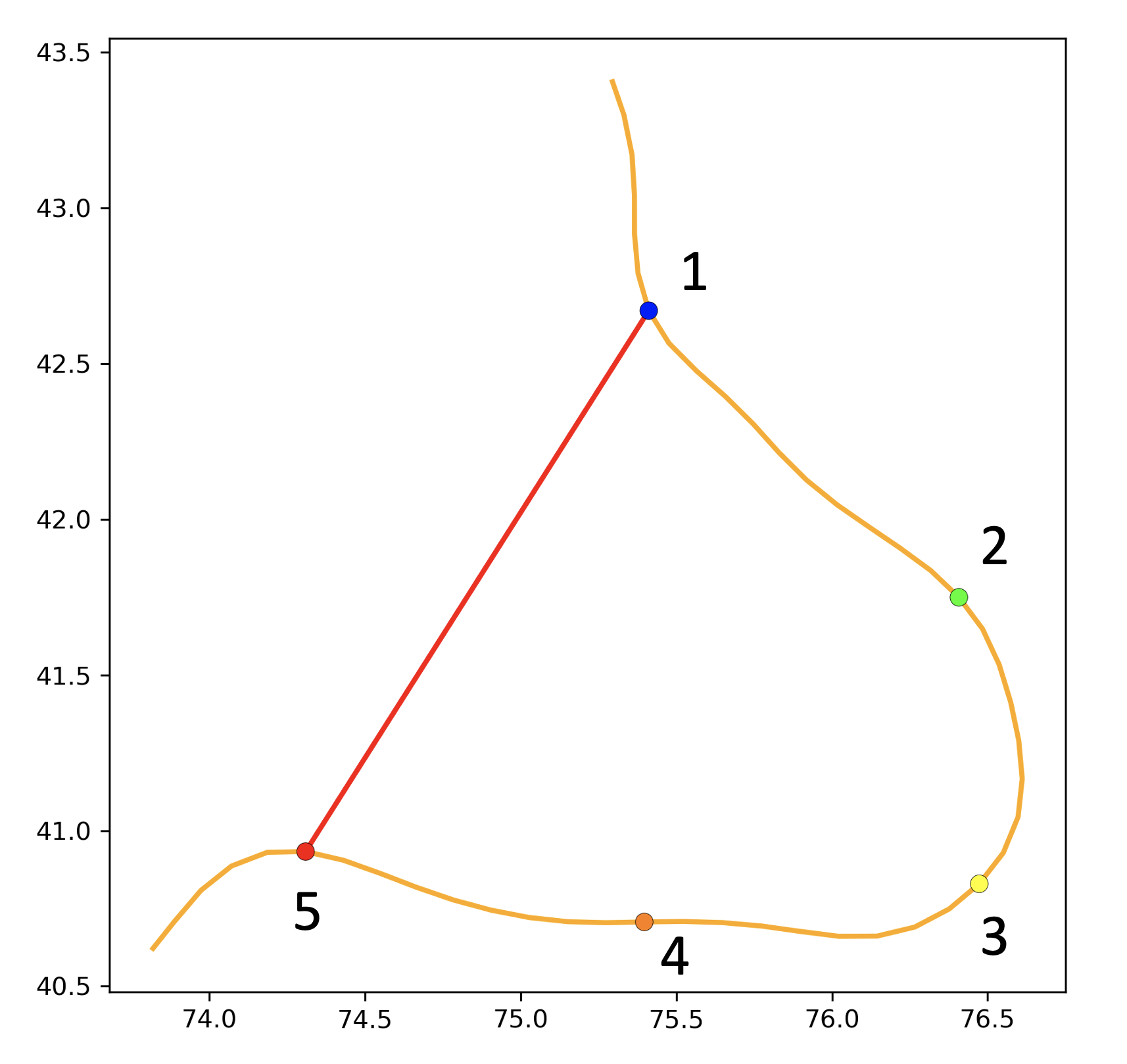}
\caption{\it Determining $5$ landmarks on a typical microtubules filament at mathematically characteristic locations for buckle structure. Landmarks 1 and 5 maximize the ratio between Euclidean distance and curve length between them times the fourth root of the maximal normal distance along the buckle, which gives the 3rd landmark. Landmark 2 is furthest from this normal line (on the side of Landmark 1). Landmark 4 maximizes normal distance from the line connecting Landmarks 5 and 3. Afterwards, all landmarks (2,3,4) have been shifted minimally to make them more equidistant .
}
\label{fig:buckles}
\end{figure}

With the theory and methods presented in this paper, we were able to shed light on developing a new method to distinguish cells with and without IFs. 

 Using the \emph{filament sensor} from \citet{EltznerWollnikGottschlichHuckemannRehfeldt2015} (see also \citet{hauke2023filamentsensor}), from cells provided by the Cellular Biophysics Research Group of the Institute for X-ray Physics at the 
University of G\"ottingen, we have extracted planar microtubules buckle structures 
and placed 5 mathematically defined landmarks on them (Figure~\ref{fig:buckles}), leading to two groups. The first group contains $65$ buckles from cells with intermediate filaments (generating stiffness) and the second group contains $29$ buckles from cells without intermediate filaments. Because the buckles' shape are considered to be the same under reflection and traversing in reverse orientation, the two groups of buckles are modeled in the reverse labeling reflection shape space $\cR\cR\Sigma_2^5$.  Figure \ref{fig:shapes} shows typical buckle structures. 
 
\begin{figure}[h!]
\centering
\includegraphics[width=0.5\textwidth]{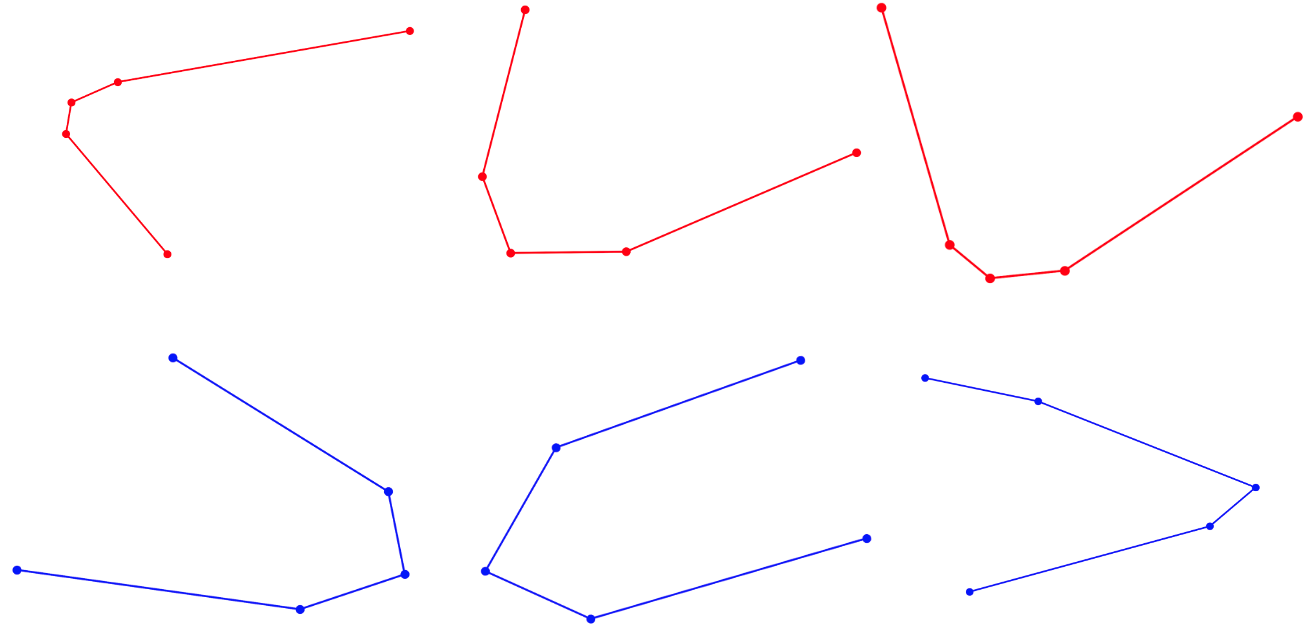}
\caption{\it Typical microtubules buckle structures discretized by 5 landmarks. Upper row: without intermediate (vimentin) filaments. Lower row: in the presence of intermediate (vimentin) filaments (generating stiffness).
}
\label{fig:shapes}
\end{figure}

 Applying the two-sample test procedures from Section \ref{scn:tests} to the data we are able to distinguish the two groups at level $\alpha = 0.05$. In fact, from the $p$-values in Table \ref{tab:buckles-test} we can see that only the test ``individual lifting'' ($T_1$) and ``individual asymmetric lifting'' ($T_2$) detect the difference, while the other two tests ``pooled lifting'' ($T$) and ``pooled lifting intrinsically'' ($T_0$) do not. Notably, after Bonferroni correction for four tests, only the test ``individual asymmetric lifting'' ($T_1$) detects the difference. For convenience, Table \ref{tab:buckles-test} also records the $p$-value of test $T_J$ from \cite{preston2010two}.

 \begin{table}[h!]		
\begin{center}
\begin{tabular}{ |c|c|c|c|c||c|} 
 \hline
 &T & $T_0$ & $T_1$ &$T_2$&$T_J$ \\ 
  \hline
 \hline
 $p$-value &0.08 & 0.052 & 0.011& 0.027&0.081 \\ 
 \hline
\end{tabular}

\end{center}
\caption{\it Results from different two-sample tests for the two groups (with and without IFs) of reverse relabeling reflection buckles. Test abbreviations as in Table \ref{tab:perfomance_shape_space}.
}
\label{tab:buckles-test}
\end{table}


\section*{Acknowledgments}
All four authors are very grateful to Alexander Lytchak for very helpful comments in differential geometry. Also, all authors thank Sarah K\"oster for providing the microtubules data.  DTV and SP gratefully acknowledge funding by DFG HU 1575/7. SP acknowledges support by the Research Foundation – Flanders (FWO) via the Odysseus II programme no. G0DBZ23N. BE acknowledges funding by DFG SFB 1456.

\bibliographystyle{Chicago}
\bibliography{shape,
zytos,
diffgeo,
comp_vis,
stats
}
\end{document}